\documentclass[11pt]{article}

\usepackage[preprint, nonatbib]{neurips_2022}

\usepackage{enumitem}
\usepackage[utf8]{inputenc} 
\usepackage[T1]{fontenc}    
\usepackage{hyperref}       
\usepackage{url}            
\usepackage{booktabs}       
\usepackage{amsfonts}       
\usepackage{nicefrac}       
\usepackage{microtype}      
\usepackage{xcolor}         
\usepackage{wrapfig}
\usepackage{pifont}

\usepackage[utf8]{inputenc}
 
\usepackage{xargs}                      

\usepackage[textsize=tiny]{todonotes}

\usepackage{amssymb, amsmath, amsthm, latexsym}
\usepackage{color}
\usepackage{hyperref}
\usepackage{graphicx}
\usepackage{array}
\usepackage{multirow}
\usepackage{tabularx}
\usepackage{wrapfig}
\usepackage{comment}
\usepackage{tablefootnote}
\usepackage{pdfpages}
\usepackage{threeparttable}
\usepackage{epstopdf}
\AppendGraphicsExtensions{.pdf}

\usepackage{authblk}
\usepackage{algorithm,algorithmic}
\usepackage{subcaption} 

\newtheorem{theorem}{Theorem}
\newtheorem{lemma}{Lemma}

\newtheorem{assumption}{Assumption}

\newtheorem{innercustomthm}{Theorem}
\newenvironment{customthm}[1]
  {\renewcommand\theinnercustomthm{#1}\innercustomthm}
{\endinnercustomthm}
  
\newtheorem{innercustomlemma}{Lemma}
\newenvironment{customlemma}[1]
  {\renewcommand\theinnercustomlemma{#1}\innercustomlemma}
{\endinnercustomlemma}
  
\newtheorem{innercustomassum}{Assumption}
\newenvironment{customassum}[1]
  {\renewcommand\theinnercustomassum{#1}\innercustomassum}
{\endinnercustomassum}  

\theoremstyle{definition}
\newtheorem{definition}{Definition}

\theoremstyle{remark}

\usepackage{kamzolov}

\newcommand{\V}{\widetilde{V}}
\newcommand{\Y}{\widetilde{Y}}
\newcommand{\B}{\widetilde{B}}

\newcommand{\C}{\mathcal{C}}
\newcommand{\mathcQ}{\mathcal{Q}}
\newcommand{\bY}{\widetilde{Y}}

\newcommand{\eqdef}{:=}

\newcommand{\EE}[2]{{\mathbb E}_{#1}\left[#2\right] } 

\newcommand{\wlambda}{\widetilde{\lambda}}
\newcommand{\cmark}{\ding{51}}%
\newcommand{\xmark}{\ding{55}}%

\title{FLECS: A Federated Learning Second-Order Framework via Compression and Sketching}

\author[1,2]{Artem Agafonov} \author[2,1]{Dmitry Kamzolov} \author[3]{Rachael Tappenden} \author[1,4,5]{Alexander Gasnikov} \author[2]{Martin Tak\'a\v{c}}
\affil[1]{Moscow Institute of Physics and Technology, Dolgoprudny, Russia}
\affil[2]{Mohamed bin Zayed University of Artificial Intelligence, Abu Dhabi, UAE}
\affil[3]{University of Canterbury, Christchurch 8041, New Zealand}
\affil[4]{Institute for Information Transmission Problems, Moscow, Russia}
\affil[5]{HSE University, Moscow, Russia}

\usepackage[textsize=tiny]{todonotes}

\begin{document}

\setlength{\abovedisplayskip}{3pt}
\setlength{\belowdisplayskip}{3pt}

\maketitle

\begin{abstract}

Inspired by the recent work FedNL (Safaryan et al, FedNL: Making Newton-Type Methods Applicable to Federated Learning), we propose a new communication efficient second-order framework for Federated learning, namely FLECS. The proposed method reduces the high-memory requirements of FedNL by the usage of an L-SR1 type update for the Hessian approximation which is stored on the central server. A low dimensional `sketch' of the Hessian is all that is needed by each device to generate an update, so that memory costs as well as number of Hessian-vector products for the agent are low. Biased and unbiased compressions are utilized to make communication costs also low. Convergence guarantees for FLECS are provided in both the strongly convex, and nonconvex cases, and local linear convergence is also established under strong convexity. Numerical experiments confirm the practical benefits of this new FLECS algorithm.


\end{abstract}

\section{Introduction}
    This paper presents a novel second-order method for Federated Learning called FLECS. In this setup, multiple agents/devices/workers collaborate to solve a machine learning problem by communicating over a central node (server). It is assumed that the central node is a large and powerful machine, but the agents/devices are much smaller (for example, a cellphone, or laptop). Moreover, raw data is stored locally on each worker and cannot be exchanged or transferred. The finite sum minimization, or Federated Learning, problem can be formalized in the following way:
    \begin{equation}\label{eq:problem}
        \textstyle{\min}_{w \in \R^d} \lb F(w):= \tfrac{1}{n}\textstyle{\sum}_{i=1}^n f_i(w)\rb,
    \end{equation}
    where $f_i(w)$ is the loss function (parametrized by $w \in \R^d$) associated with the data stored on the $i$-th machine, $F(w)$ is the empirical loss function, and both the problem dimension $d$, and the number of machines $n$, can be large.
\\[2pt]
Problems of the form \eqref{eq:problem} arise not only in distributed machine learning \cite{kra13, li2014scaling}, but also in robotics, resource allocation, power system control, control of drone or satellite networks, distributed statistical inference and optimal transport, and multi-agent reinforcement learning \cite{xia06, rab04, nedic2017fast, dvurechenskii2018decentralize, uribe2018distributed, ram2009distributed}.
    \\[2pt]
    A variety of first-ordered methods have been introduced for the problem \eqref{eq:problem} under the stated computational set-up (a central server connected to multiple smaller devices). Early work includes that of \cite{Konecny2016a,Konecny2016b,McMahan2017}, which focused on distributed first order optimization algorithms that were designed with the computational structure in mind. Recent advances include works focusing on reducing the  communication burden via compression techniques \cite{alistarh2017qsgd,horvath2019natural,chen2022distributed}, including momentum for algorithm speed up \cite{Mishchenko2019,Li2020accFL}, variance reduction techniques \cite{Horvath2019}, the use of adaptive learning rates \cite{Xie2019}, and works that focus on maintaining client data privacy \cite{Kairouz2019}.
    \\[2pt]
    Beyond first-order methods, second-order methods have also been developed for Federated Learning. Some of these algorithms \cite{Zhang2018,dvurechensky2021hyperfast,daneshmand2021newton,bullins2021stochastic, agafonov21acc} utilize statistical similarity (the homogeneous data setting), which, roughly speaking, means that the local functions $f_i$ are similar, and they approximate the overall empirical risk function $F$ well. FedNL \cite{safaryan2021fednl} for the  second-order method for Federated Learning in the case of \emph{heterogeneous} data. While FedNL makes a valuable step toward the applicability of second-order methods to Federated Learning, this framework seems to be impractical for large-scaled problems, since $d \times d$ Hessian approximations are stored locally on workers and transferred to the server via compressed  matrix communications. (This violates the assumption in the current work, that the individual workers have small memory.)
    \\[2pt]
The goal here is to build upon the work in \cite{safaryan2021fednl} and develop a new second order framework for Federated Learning that is based upon more realistic assumptions about the available computational architecture. That is, we assume that individual workers are devices with relatively low memory availability and computational capacity (e.g. mobile phones), but they can connect to a large and powerful central machine. The typical situation, that downloading is much faster than uploading, is also assumed. \footnote{ 
An example of such a setting is (e.g., battery powered) wireless network \cite{rajaravivarma2003overview,
yick2008wireless}
where the cost of sending data from devices to the centralized server is energy demanding, implying severely limited communication bandwidth to the server \cite{he2021cell,chen2021fedsvrg}.
On the other hand, the server is usually connected to an energy network and has no constraints for sending data to other devices in the network.}
We introduce the FLECS second-order framework for strongly convex and non-convex optimization,
which places no additional memory on workers and reduces the number Hessian-vector products per node. Therefore, our algorithm can be viewed as the lightweight FedNL, since it provides similar theoretical guarantees and has almost the same global. FLECS uses sketching \cite{woodruff2014sketching} and matrix compression for efficient communication and allows for several Hessian learning techniques (including truncated LSR1 update) and two iterate updates, namely Truncated inverse Hessian approximation and FedSONIA, inspired by \cite{paternain2019newton} and \cite{jahani2021sonia}.

    \subsection{Contributions}\label{subsec:contribution}
    We propose FLECS: a \textbf{F}ederated \textbf{LE}arning Second-Order Framework via \textbf{C}ompression and \textbf{S}ketching with the following properties.\vspace{-2mm}
    \begin{itemize}[leftmargin=10pt,nolistsep]      
    
    \item \textbf{No additional storage on workers.}
        In large-scale machine learning problems, it is impossible to store the $d \times d$-dimensional Hessian approximations. For a dataset with $d = 50,000$ features, the local Hessian approximation requires more than $10$ GB of memory. For FLECS, the Hessian approximations are stored on the server, so memory requirements for the individual workers remain low.
        \item \textbf{Efficient communication via sketching and compression.}
        To reduce the cost of communication, we use a sketching technique. This allows the exchange of $d \times m$ matrices instead of $d \times d$, where $m \ll d$ is the user defined memory size, which could be set as $m=1$. Moreover, we exploit compression to make the algorithm even more communication-efficient. This strategy heavily improves the one in \cite{safaryan2021fednl}, since we decrease the dimension of the transferred matrices $d/m$ times. Moreover, sketching  reduces the number of Hessian-vector products from $d$ to $m$. 
        \item \textbf{Heterogeneous data setting.} Our framework is flexible and does not assume that data on the workers is independent and identically distributed.
        \item \textbf{Hessian Learning via Quasi-Newton updates.}
        Motivated by quasi-Newton methods, sketching techniques and inverse Hessian truncation, we propose the Truncated L-SR1 update and Direct update for our Hessian approximations. 
        \item \textbf{Theoretical Analysis.} Convergence guarantees in both the strongly convex, and nonconvex cases, are provided, as well as a local linear rate  of convergence under strong convexity independent of problem conditioning.
        \item \textbf{Competitive Numerical Experiments.} Our numerical experiments highlight the practical benefits of our proposed framework.
    \end{itemize}
    \begin{table}
    \centering
    \caption{Comparison between FedNL\cite{safaryan2021fednl} and FLECS.}
        \begin{tabular}{cccc}
        \toprule
         & FedNL & FedNL-LS & FLECS    \\
         &  \cite{safaryan2021fednl} & \cite{safaryan2021fednl} &  [this paper]   \\
        \midrule
        heterogeneous data setting & \cmark & \cmark & \cmark \\
        additional storage on workers & $O(d^2)$ & $O(d^2)$ & $0$\\
        supports biased and unbiased compression & \cmark & \cmark & \cmark\\
        communication cost (omitting compression) & $O(d^2)$ & $O(d^2)$ & $O(md) $\\
        Hessian-vector products per iteration on worker & $d$ & $d$ & $m$\\
        complexity per iteration on the server & $O(d^3)$ & $O(d^3)$ &\begin{tabular}{@{}c@{}}$O(d^3)$~(Alg. \ref{alg:approx_trunc}) \\ $O(nmd^2)$~(Alg. \ref{alg:approx_fedsonia}) \end{tabular}  \\
        convergence in non convex case   & \xmark & \xmark & \cmark \\
        strongly convex case: global convergence  &  \xmark &  linear & linear \\
        strongly convex case: local convergence  & superlinear & \xmark & linear (Alg. \ref{alg:approx_trunc}) \\
        \bottomrule
        \end{tabular}
        \vspace{-10pt}
    \end{table}
    \vspace{-10pt}
   \paragraph{Organization.}
    The remainder of the paper is organized as follows. Related works are described in Section \ref{sec:related}. Section \ref{sec:algo} presents the proposed framework, including all algorithmic details. In Section~\ref{sec:theory}, we establish theoretical convergence guarantees (all proofs can be found in the appendix). Numerical experiments are provided in Section \ref{sec:experiments} (additional experiments are provided in the appendix), and concluding remarks and potential future research directions are given in Section~\ref{sec:conclusion}.
    
\vspace{-2mm}

    \section{Related Work}\label{sec:related}
    
    Second-order methods are well established in the optimization literature. Such methods employ the Hessian (or an approximation to it) and thereby important (possibly approximate and/or partial) curvature information is available for utilization. The benefits include `better' search directions, and faster rates of convergence that may be independent of the condition number, although the algorithms often have higher memory and computational requirements than first order methods. To balance the trade off between improved convergence rates versus increased costs, approaches that involve modifying or approximating the Hessian have been proposed, and several of these are discussed now.
\\[2pt]    
    Arguably, some of the most popular algorithms that use approximations to the Hessian are quasi-Newton methods. They employ information from the iterates to build an approximation to the Hessian that improves in accuracy as the algorithm progresses. Among these are algorithms that use BFGS, SR1, and DFP updates \cite{fletcher2013practical, nocedal1999numerical, dennis1977quasi}, which are now considered to be classics in optimization due to their effectiveness and practicality, and also because, under certain conditions,  theoretical results related to the accuracy of the Hessian approximations are available. This motivates the current work, which incorporates quasi-Newton type updates to learn a good approximation to the Hessian. 
    \\[2pt]
    Of particular interest is the SR1 update, which does not guarantee positive definiteness of the Hessian approximation \cite{dennis1977quasi}. Previously, this was considered to be a disadvantage, but in the case of non-convex optimization, it would appear to be advantageous. L-SR1, a limited memory variant of SR1, belongs to the class of limited memory quasi-Newton algorithms \cite{lu1996study, brust2017solving}. These methods store only the last $m$ gradient differences $Y_k = [y_{k-m+1}, \ldots, y_k]$ and iterate differences $S_k = [s_{k-m+1}, \ldots, s_k]$, where $m$ is the memory size, $y_k = \nabla F(w_k) - \nabla F(w_{k-1})$ and $s_k = w_k - w_{k-1}$. The compact form of the L-SR1 update is \cite{byrd1994representations}\vspace{-2mm}
    \begin{equation}
        B_k = B_0 + (Y_k - B_0S_k)M_k^{-1}(Y_k - B_0S_k)^T,
    \end{equation}
    where $M_k = L_k + D_k + L_k^T - S_k^TB_0S_k$, $~S_k^TY_k = L_k + D_{k} + U_k$, $D_{k}$ is diagonal, $L_k$, $U_k$ are strictly lower and upper triangular, and $B_0$ is the initial approximation.
    \\[2pt]
    The work \cite{berahas2021quasi} proposed the use of a sampled $S_k$, rather than one based upon the previous $m$ iterate differences. The idea is, given the current iterate $w_k$, simply sample $m$ random directions, and use the differences with $w_k$ as the columns of $s_k$ (forming $Y_k$ appropriately). An $S_k$ with this form can be viewed as sketching matrix \cite{woodruff2014sketching}, and the Newton Sketch algorithm in \cite{pilanci2017newton}, shows that even randomly projected or a sub-sampled Hessian provides enough curvature information for fast convergence. 
    \\[2pt]
    The Nonconvex Newton method of \cite{paternain2019newton} is a different approach, which is an adaptation of Newton's method for nonconvex problems. At each iteration, the singular value decomposition of the Hessian is computed, small eigenvalues (in absolute value) are truncated, and then a positive definite modification of the Hessian is constructed. This is then used to create Newton-type direction. 
    \\[2pt]
    SONIA \cite{jahani2021sonia} bridges the gap between first-order and second-order methods, combining the ideas of quasi-Newton algorithms, sketching \cite{pilanci2017newton} and truncation \cite{paternain2019newton}. It constructs the search directions using curvature information in one subspace and gradient information in the orthogonal complement.
    \\[2pt]
    Second-order methods for Federated Learning, can generally be divided in two groups, based upon whether a homogeneous or heterogeneous data setting is assumed. Algorithms in the first group rely on the concept of statistical similarity, which, in practice, means that each local function $f_i$ is a `good' approximation of the global objective function $F$. (More formally, $\|\nabla^2 F - \nabla^2 f_i\| \leq \beta$ for some $\beta >0$.) Usually, in this case, methods utilize the part of the data stored on the server and use it to approximate the full Hessian \cite{Zhang2018,dvurechensky2021hyperfast,daneshmand2021newton,bullins2021stochastic, agafonov21acc}. 
    \\[2pt]
    No such statistical similarity  (i.i.d.) assumption is made in the heterogeneous setting, which is perhaps more realistic (e.g., it may be impossible/impractical to have local data (on a mobile phone say), that approximates the data on the server well). The main work in this setting is FedNL \cite{safaryan2021fednl}. The authors proposed a way to learn Hessian approximations and then take an approximate Newton-type step. However, the main drawback is that the Hessian approximations are stored locally and communications performed via compressed $d\times d$ matrices. For large-scale optimization problems, it is impossible to store the $d \times d$ matrix locally, because of memory limitations.
    Extensions of FedNL can be found in \cite{islamov2021distributed}, \cite{qian2021basis}.
    \\[2pt]
    The approach proposed here combines ideas from quasi-Newton methods, sketching, and Hessian learning, and develops a second-order Federated Learning framework with efficient communications and without an additional memory burden on workers. All Hessian approximations are stored on the server, not the workers. Communications are based on a user-defined memory hyperparameter $m$ and utilize compression to reduce communication. 
    Sketching  allows to transfer $d \times m$ matrices and reduce the number of Hessian-vector products on the worker.
    The proposed framework allows for two Hessian learning techniques and two iterate update rules. The Truncated Inverse Hessian approximation is based on the idea of truncating eigenvalues (as in \cite{paternain2019newton}) of Hessian approximate, while FedSONIA is based upon SONIA \cite{jahani2021sonia} and reduces the iteration complexity of the step from $O(d^3)$ to $O(nmd^2)$. 
    
\section{FLECS}\label{sec:algo}
    Inspired by the works \cite{jahani2021sonia, safaryan2021fednl}, we propose a general framework for second-order methods for federated learning problems \eqref{eq:problem} that allows for different forms for the iterates and different Hessian approximations/updates. One of main goals of this framework is to create fast communication-efficient second-order methods that do not place additional memory requirements on the workers. 
\\[2pt]    
    To make broadcasting between the server and nodes cheap we employ sketching and compression. Throughout the paper, $S_k \in \R^{d \times m}$ denotes a sketching matrix, where $m$ is the memory size ($m \ll d$). Any operator that compresses a matrix is denoted by $\mathcal{C}: \R^{d\times m} \to \R^{d\times m}$ 
    , e.g. random dithering, Rank-$R$, Top-$K$, Rand-$K$ (description of compressors can be found in Appendix \ref{app:compressors}). 
    The Hessian approximation of the $i$-th worker at iteration $k$ is denoted by $B_k^i$, and is stored on the central node.
    \\[2pt]
    Our framework can be described in three steps (further details follow in later subsections):
\begin{enumerate}[leftmargin=15pt,noitemsep,topsep=-5pt]
        \item \textbf{Worker step.} Each worker receives the current iterate $w_k$, and their local Hessian approximation multiplied by a sketching matrix, i.e., the product $B_k^i S_k \in \R^{d \times n}$. The matrix is sent directly (i.e., no compression), since, typically, downloading is faster than uploading. Each worker then computes their local Hessian-sketch product $Y_k^i \eqdef \nabla^2 f_i(w_k) S_k$ and sends back the compressed difference $C_k^i \eqdef \C(Y_k^i - B_k^iS_k) \in \R^{d \times m}$.  
        \item \textbf{Hessian approximation update.} The server receives $C_k^i$, $M_k^i \eqdef S_k^T(\nabla^2 f_i(w_k))S_k$, and $\nabla f_i(w_k)$ from all $i = 1, \ldots, n$ workers. Using this information, as well as the current (local) Hessian approximation $B_k^i$, the new approximation $B_{k+1}^i$ is formed. Our framework allows $B_{k+1}^i$ to be computed by a (1) Truncated L-SR1 update, or (2) Direct update. 
        \item \textbf{Iterate update.} The server forms the matrices 
        $$B_k \eqdef  \tfrac{1}{n}\textstyle{\sum}_{i=1}^n B_k^i, ~~ \nabla F(w_k), ~~ M_k = \frac{1}{n}\textstyle{\sum} _{i=1}^n M_k^i,$$
        and calculates the new iterate $w_{k+1}$ via update rule:
        \begin{equation}\label{eq:update_rule}
            w_{k+1} = w_k + \alpha_k p_k
        \end{equation}
        where $\alpha_k > 0$ is the step-size. There are several options for computing the search direction $p_k$, including (1) a truncated inverse Hessian approximation step, or (2) via a SONIA step for Federated Learning (FedSONIA).
        The process then repeats from Step 1.
    \end{enumerate}
    
    \subsection{Initialization}
    The algorithm is initialized with a user defined memory size $m \ll d$, and $n$ matrices $B_0^i \in \R^{d \times d}$ are stored on the central node, each representing an approximation to the local Hessian for the $i$th worker. Various additional parameters are initialized depending on the selected update rule (to be described later). At iteration $k > 0$ of FLECS, $m$ directions $\lb s_1, \ldots, s_m\rb$ are randomly sampled at all nodes (including the server) and combined together, forming the columns of a sketching matrix $S_k \in \R^{d \times m}$. Note that $S_k$ is the same for all machines and the main node; this is \emph{guaranteed} by setting the random seed to be equal to the iteration number $k$ during sampling.
    
    \subsection{Worker Step}\label{sec:workerstep}
    Each of the workers receives $B_k^iS_k$, and the iterate $w_k$, from the server. On each node we compute 
    \begin{equation}\label{eq:node_results}
        C_k^i \eqdef \C(Y_k^i - B_k^iS_k), ~ \nabla f_i(w_k), ~~ M_k^i \eqdef S_k^TY_k^i = S_k^T(\nabla^2 f_i(w_k))S_k, ~~ i = 1, \ldots, n.
    \end{equation}
    The quantities in \eqref{eq:node_results} are used to construct the Hessian approximation on the main node. Since $M_k^i$ is an $m \times m$, it can be sent directly without compression. So too can the vector $\nabla f_i(w_k)$. For the matrix $Y_k^i$ we use the compression operator $\C$. In order to make the variance of this compression low, we use an error-feedback technique and send $C_k^i = \C(Y_k^i - B_k^iS_k).$
    \subsection{Hessian Approximation update}
    The server receives \eqref{eq:node_results} from each worker $i=1, \ldots, n$. Firstly, we restore an approximation to $Y_k^i$ as 
    \begin{equation}\label{eq:Y_tilda}
        \Y_k^i = C_k^i + B_k^iS_k.
    \end{equation}
    Then, we apply one of the following options to update the Hessian approximation.
     
    \textbf{Hessian update option 1. Truncated L-SR1 update (Algorithm \ref{alg:hess_lsr1}).} 
    In this case we compute the approximate Hessian of the $i$-th worker using the L-SR1 update:
    \begin{equation}\label{eq:LSR1}
        B_{k+1}^i = B_k^i + (\bY_k - B_k^iS_k)(M_k^i - S_k^T B_k^i S_k)^\dagger (\bY_k - B_k^iS_k)^T.
    \end{equation}
    To ensure numerically stability of the update \eqref{eq:LSR1}, truncation is employed. At first, we take the spectral decomposition of the small $(m\times m)$ matrix $$M_k^i - S_k^T B_k^i S_k = U_k^iL_k^i(U_k^i)^T,$$ where $L_k^i \in \R^{d\times d}$ is a diagonal matrix containing the eigenvalues, and $U_k^i$ is an orthogonal matrix containing the corresponding eigenvectors.
    Next, we compute the inverse $(L_k^i)^{-1}$, and truncate small values. In particular, define $(l_k^i)_j^{-1} \eqdef (L_k^i)^{-1}_{jj}$, and set all elements such that $|(l_k^i)_j^{-1}| \leq \omega$ to be $0$, where $\omega>0$ is a user defined parameter that must be initialized. The truncated matrix is denoted by $[(L_k^i)^{-1}]_\omega$. Finally, the update \eqref{eq:LSR1} becomes
    \begin{equation}\label{eq:LSR1_truncated}
        B_{k+1} = B_k + (\bY_k - B_k^iS_k)U_k^i [(L_k^i)^{-1}]_\omega  (U_k^i)^T(\bY_k - B_k^iS_k)^T.
    \end{equation}
    \begin{algorithm}[t]
          \caption{Truncated L-SR1 update}\label{alg:hess_lsr1}
            \begin{algorithmic}[1]
            \REQUIRE $\Y_k^i \in \R^{d \times m}, ~M_k^i \in \R^{m \times m}, ~B_k^i \in \R^{d \times d}, ~S_k \in \R^{d \times m}$ for $i = 1, \ldots, n$,
            $~\omega > 0$ -- truncation constant.
            \STATE \textbf{On the server:}
            \FOR{$i = 1, \ldots, n$}
            \STATE compute $(M_k^i - (S_k^i)^T\Y_k^i) =U_k^i L_k^i (U_k^i)^T$;
            \STATE truncate $(L_k^i)^{-1}$  
            to form $[(L_k^i)^{-1}]_\omega$;
            \STATE compute $B_{k+1}^i$ via \eqref{eq:LSR1_truncated}.
            \ENDFOR
        \end{algorithmic}
    \end{algorithm}
\textbf{Hessian update option 2. Direct update (Algorithm \ref{alg:hess_direct}).}
    Setting $B_0=0$ in \eqref{eq:LSR1} gives
    \begin{equation}\label{eq:SONIA_approx}
        \B^i_k = \Y^i_k (M^i_k)^\dagger (\Y^i_k)^T,
    \end{equation}
    which we refer to as the Direct update. Then, we utilize the following learning technique
    \begin{equation}\label{eq:direct_update}
        B_{k+1}^i = (1 - \beta_k)B_k^i +  \beta_k\B^i_k,
    \end{equation}
    where $0<\beta_k\leq 1$ is the learning rate. Thus, for the Direct update, the new Hessian approximation is a convex combination of the previous update and the matrix in \eqref{eq:SONIA_approx}.
    \\[2pt]
    Note, that for both Option 1 and Option 2, the Hessian approximation $B_{k+1}^i$ is symmetric, since $M_k^i = S_k^T\nabla f_i(w_k)S_k$, $B_k^i$, and $M_k^i - S_k^TB_k^iS_k$ are symmetric for $i = 1, \ldots, n$. 
    
    
    \begin{algorithm}[h]
          \caption{Direct update.}\label{alg:hess_direct}
            \begin{algorithmic}[1]
            \REQUIRE $\Y_k^i \in \R^{d \times m}, ~M_k^i \in \R^{m \times m}, ~B_k^i \in \R^{d \times d}$ $\forall i$
            \STATE \textbf{On the server:}
            $0<\beta_k\leq 1$ -- learning rate.
            \FOR{$i = 1, \ldots, n$}
            \STATE compute $\B^i_k = \Y^i_k (M^i_k)^\dagger (\Y^i_k)^T$;
            \STATE select learning rate $\beta_k$ 
            \STATE compute
            $B^i_{k+1} = (1 - \beta_k)B^i_{k} + \beta_k\B^i_k$.
            \ENDFOR
        \end{algorithmic}
     \end{algorithm}

    \subsection{Iterate update}
    Note that, at the start of the iterate update step, all the Hessian approximations $B_{k+1}^i$, and all variables from \eqref{eq:node_results}, are available on the main node. Then, the server forms
    \begin{align*}
          \nabla F(w_k) &= \tfrac{1}{n} \textstyle{\sum} _{i=1}^n f_i(w_k),  &
          \widetilde{Y}_k &\eqdef \tfrac{1}{n}\textstyle{\sum} \Y_k^i = \tfrac{1}{n}\textstyle{\sum} _{i=1}^n (C_k^i + H_k^i S_k), \\
          M_k &\eqdef \tfrac{1}{n} \textstyle{\sum}_{i=1}^n M_k^i = S_k \nabla^2 F(w_k)S_k^T, 
          &
          B_{k+1} &\eqdef \tfrac{1}{n} \textstyle{\sum}_{i=1}^n B_{k+1}^i.
    \end{align*}
    Now the goal is to compute a search direction $p_k$, of the general form:
    \begin{equation}\label{eq:search_dir}
        p_k = - A_k \nabla F(w_k),
    \end{equation}
    for some matrix $A_k$. Two specific options for $A_k$ are described now.
     
    \textbf{Search direction option 1. Truncated inverse Hessian approximation (Algorithm \ref{alg:approx_trunc}).}
    Here, $A_k$ is taken to be the truncated inverse Hessian approximation $\left(|B_{k+1}|_\omega^\Omega\right)^{-1}$. Consider the following. 
    
    \begin{definition}\label{def:truncation}
        Let $B_k, V_k, \Lambda_k$ be matrices such that $B_k = V_k \Lambda_k V_k^T$, and let $0 < \omega \leq \Omega$. The truncated inverse Hessian approximation of $B_{k}$ is $\left(|B_{k}|_\omega^\Omega\right)^{-1} := V_k (|\Lambda_k|_\omega^\Omega)^{-1}V_k^T$,
        where \\ $(|\Lambda_k|_\omega^\Omega)_{ii} = \min\lb \max \lb |\Lambda|_{ii}, \omega\rb, \Omega\rb.$
    \end{definition}
    Definition \ref{def:truncation} was proposed in \cite{paternain2019newton} and was used to provide a convergence guarantee for their Nonconvex Newton method (to a local minimum). Firstly, an eigen-decomposition of $B_k$ is computed, but with every eigenvalue replaced by its absolute value. Secondly, a thresholding step is applied, so that any eigenvalue (in absolute value) that is smaller (resp. greater) than a user defined threshold $\omega$ (resp. $\Omega$) is replaced by $\omega$ (resp. $\Omega$). This truncation is crucial for several reasons. Firstly, it ensures $B_k$ is well-defined, which leads to stable updates. Secondly, it ensures that after truncation, the resulting matrix will be full-rank and positive definite. Therefore, it guarantees that the following search direction is a \emph{decent} direction
    \begin{equation}\label{eq:search_dir_sr1}
        p_k = \left(|B_{k+1}|_\omega^\Omega\right)^{-1} \nabla F(w_k).
    \end{equation}
    \begin{algorithm}[t]
           \caption{Truncated inverse Hessian approximation}\label{alg:approx_trunc}
            \begin{algorithmic}[1]
            \REQUIRE $\nabla F(w_k) \in \R^d, ~ \bY_k \in \R^{d\times m}, ~ M_k \in \R^{m\times m}, ~ B_{k+1} \in \R^{d\times d}$,
            $\Omega > \omega > 0$ -- truncation constants.
            \STATE \textbf{On the server:}
            \STATE compute spectral decomposition $B_{k+1} = V_k \Lambda_k V_k^T$;
            \STATE truncate $\Lambda_k$ 
            to form $|\Lambda_k|_\omega^\Omega$  via Definition \ref{def:truncation};
            \STATE compute search direction $p_k$ via \eqref{eq:search_dir_sr1};
            \RETURN $p_k$.
        \end{algorithmic}
    \end{algorithm}

    \textbf{Search direction option 2. FedSONIA (Algorithm \ref{alg:approx_fedsonia}).}
    FedSONIA is a modification of SONIA \cite{jahani2021sonia} that is adapted to the Federated Learning setting. The key idea of SONIA is to utilize curvature information in one subspace and perform a gradient step in orthogonal complement. The intuition is that, by incorporating some curvature information the resulting search direction should be `better' than a pure gradient direction, but the user has control over the dimension of each subspace, and therefore they have control over the cost of obtaining the curvature information. Truncation (Definition \ref{def:truncation}) and modified SR1 update (eq. \eqref{eq:LSR1}) ensure that the Hessian-approximation is full-rank, positive and well defined, which is crucial to avoid checking condition on the curvature pairs $Y_k^TS_k$ ($\Y_k^TS_k$ for FedSONIA). To construct the approximation of the Hessian we use \eqref{eq:SONIA_approx}. Note, that SONIA does not need the Hessian approximation $B_k$ itself to commit a step. Instead, it utilizes only Hessian-matrix product $Y_k$ and sketching matrix $S_k$.
    \\[2pt]
    Define $B_{k + 1}^{\text{SONIA}} = \Y_k (M_k)^\dagger \Y_k^T$. Following the strategy in \cite{jahani2021sonia}, apply an economy $QR$ factorization of $\Y_k = Q_k R_k$, so that $Q_k \in \R^{d \times m}$ has orthonormal columns and $R_k \in \R^{m\times m}$ is upper triangular. Thus,
    $B_{k + 1}^{\text{SONIA}} = Q_kR_k(M_k)^\dagger R_k^TQ_k^T$. 
    Applying the spectral decomposition, we obtain
    $R_k(M_k)^\dagger R_k^T = V_k \Lambda_k V_k^T$, where the columns of $V_k \in \R^{m\times m}$ form an orthonormal basis (of eigenvectors), and $\Lambda_k \in \R^{m\times m}$ is a diagonal matrix containing the corresponding eigenvalues. Therefore, 
    $$B_{k + 1}^{\text{SONIA}} = Q_k V_k \Lambda_k V_k^T Q_k^T \stackrel{ \widetilde{V}_k := Q_k V_k}{=}\widetilde{V}_k \Lambda_k \widetilde{V}_k^T,$$ 
    where $\widetilde{V}_k \in \R^{d \times m}$ has orthonormal columns since $Q_k$ has orthonormal columns and $V_k$ is an orthogonal matrix. 
    Then, we apply Definition \ref{def:truncation} to form the truncated approximation 
    \begin{equation}\label{eq:A_sonia}
        \left(|B_{k+1}^{\text{SONIA}}|_\omega^\Omega\right)^{-1} = \V_k(|\Lambda_k|_\omega^\Omega)^{-1}\V_k^T.
    \end{equation} 
    Next, we perform subspace decomposition on the current gradient direction:
    \begin{equation}\label{eq:decomposition}
        \nabla F(w_k) = g_k + g_k^\perp,
    \end{equation}
    where $g_k = \widetilde{V}_k\widetilde{V}_k^T\nabla F(w_k)$ and $g_k^\perp = (I - \widetilde{V}_k\widetilde{V}_k^T)\nabla F(w_k)$. Putting this all together, the FedSONIA search direction (similar to \cite{jahani2021sonia}) is:
    \begin{equation}
        \label{eq:search_direction_sonia}
        p_k :=  - \left(|B_{k+1}^{\text{SONIA}}|_\omega^\Omega\right)^{-1} g_k - \rho_k g^\perp_k \stackrel{\eqref{eq:A_sonia}, ~ \V_k^TV_k = 1}{=} -\left(\left(|B_{k+1}^{\text{SONIA}}|_\omega^\Omega\right)^{-1} + \rho_k (I-\widetilde{V}_k\widetilde{V}_k^T)\right)\nabla F(w_k),
    \end{equation}
         where $\rho_k \in [1/\Omega, 1/(\max\limits_i[(|\Lambda_k|_\omega^\Omega)^{-1}]_{ii}) ]$. One has control over the cost of the curvature information in \eqref{eq:search_direction_sonia} via the memory size $m$. Note that this will also affect the communication complexity. 
         \\[2pt]
    \begin{algorithm}[t]
           \caption{FedSONIA}\label{alg:approx_fedsonia}
            \begin{algorithmic}[1]
            \REQUIRE $\nabla F(w_k) \in \R^d, ~ \bY_k \in \R^{d\times m}, ~ M_k \in \R^{m\times m}$,
            $\Omega > \omega > 0$ -- truncation constants.
            \STATE \textbf{On the server:}
            \STATE compute $B_k^{\text{SONIA}} \eqdef \Y_k (M_k)^\dagger \Y_k^T$;
            \STATE compute $QR$ factorization of $\Y_k (= Q_k R_k)$;
            \STATE compute spectral decomposition of $R_k(M_k)^\dagger R_k^T (= V_k \Lambda_k V_k^T)$;
            \STATE construct $\widetilde{V}_k := Q_k V_k$;
            \STATE truncate $\Lambda_k$ 
            to form $|\Lambda_k|_\omega^\Omega$ via Definition \ref{def:truncation};
            \STATE Set $\rho_k$ and decompose gradient via \eqref{eq:decomposition};
            \STATE Compute search direction $p_k$ via \eqref{eq:search_direction_sonia};
            \RETURN $p_k$.
        \end{algorithmic}
    \end{algorithm}
Finally, a new search direction $p_k$ is computed, which depends on the chosen update strategy (Algorithm \ref{alg:approx_trunc} or \ref{alg:approx_fedsonia}) and to compute the new iterate $w_{k+1}$ via \eqref{eq:update_rule}. Recall that for both these strategies the resulting search direction $p_k$ has the form \eqref{eq:search_dir}, 
    where $A_k = (|B_{k+1}|_\omega^\Omega)^{-1}$ for the Truncated inverse Hessian approximation \eqref{eq:search_dir_sr1}, and  $A_k = \left(|B_{k+1}^{\text{SONIA}}|_\omega^\Omega\right)^{-1} + \rho_k (I-\widetilde{V}_k\widetilde{V}_k^T)$ for FedSONIA \eqref{eq:search_direction_sonia}. The process then restarts (from Section~\ref{sec:workerstep}). Our framework is summarized in Algorithm~\ref{alg:main}. 
\\[2pt]    
{\bf Complexity.} The per iteration complexity of FLECS depends on the chosen options for both the Hessian approximation and the iterate updates. The complexity of the Truncated L-SR1 update for one Hessian approximation consists of matrix products $(O(md^2))$ and a singular value decomposition of $m \times m$ matrix $(O(m^3))$. For the Direct update, the matrix products cost $O(md^2)$ and forming the pseudo-inverse of the $m \times m$ matrix costs $O(m^3)$. Therefore, the total complexity of either Hessian approximation update is $O(nmd^2)$. The complexity of computing search direction via Truncated inverse Hessian approximation is $O(d^3)$, since we utilize spectral decomposition and products of $d\times d$ matrices. FedSONIA update comprises 1) QR factorization $(O(dm^2))$; 2) pseudoinverse of $m\times m$ matrix $(O(m^3))$; 3) spectral decomposition of $m \times m$ matrix $(O(m^2))$; which gives total complexity $O(dm^2)$. Finally, the complexity of FLECS on the server only depends on options for the search direction and equals $O(d^3 + nmd^2)$ for Truncated Inverse Hessian approximation and $O(nmd^2)$ for FedSONIA. The the worker step's complexity consists of $m$ Hessian-vector products  and matrix multiplication $(O(md^2))$. 
\begin{algorithm}[t]
           \caption{FLECS}\label{alg:main}
            \begin{algorithmic}[1]
            \REQUIRE $w_0$~-- starting point,
            $m$~-- memory size, $B_0^i$~-- initial Hessian approximations for each worker $i=1\ldots n$ on the server, $0 < \omega < \Omega$ -- truncation constants.
            \FOR{$k=0,1,\ldots$}
            \STATE \textbf{On $i$-th machine:}
            \STATE sample $S_k \in \R^{d\times m}$;
            \STATE collect $H_k^iS_k, ~ w_k$ from the server;            \STATE compute $Y_k^i:= \nabla^2f_i(w_k)S_k, ~ U_k^i:= S_k^T Y_k^i$;
            \STATE send $\nabla f_i(w_k),~ M_k^i,~ C_k^i = \mathcal{C}(Y_k^i - H_k^i S_k)$ to the server.
            \STATE \textbf{On the server:}
            \STATE sample $S_k$;
            \STATE collect $C_k^i,~ M_k^i, ~ i=1\ldots n$ from workers;
            \STATE compute $\bY_k^i = C_k^i + H_k^iS_k$;
            \STATE compute $B_{k+1}^i$ via Algorithm \ref{alg:hess_lsr1} or select learning rate $\beta
            _k$ and compute  $B_{k+1}^i$ via Algorithm \ref{alg:hess_direct};
            \STATE form $\nabla F(w_k), \bY_k, M_k, B_{k+1}$ as average over workers of $\nabla f_i(w_k), \bY_k^i, M^i_k, B^i_{k+1}, ~ i=1, \ldots, n$;
            \STATE compute search direction $p_k$ via Algorithm \ref{alg:approx_trunc} or \ref{alg:approx_fedsonia};
            \STATE select stepsize $\alpha_k$ and set $w_{k+1} = w_k + \alpha_k p_k$;
            \STATE sample $S_{k+1} \in \R^{d\times m}$;
            \STATE send $w_k, B_{k+1}S_{k+1}$ to all workers.
            \ENDFOR
        \end{algorithmic}
    \end{algorithm}

    \section{FLECS: Convergence Theory} \label{sec:theory}
    
    In this section, convergence theory for our FLECS framework is presented. We provide global convergence results in both the strongly convex, and nonconvex cases. A local convergence result is also presented, which shows local linear convergence of the iterates under strong convexity. We define $w_0$ to be the initial point, $w^{\star} = \arg \min \limits_{w \in \R^d} F(w)$ and $F^{\star} = F(w^{\star})$.
    \vspace{-5pt}
    \begin{assumption}\label{assum:diff}
        The function $F(w)$ is twice continuously differentiable for all $w \in \R^d$.
    \end{assumption}
    \vspace{-5pt}
    We begin by listing the preliminary lemma.
    \begin{lemma} \label{lem:positive_df} 
        Let Assumption \ref{assum:diff} hold. Let $A_k$ be defined as in \eqref{eq:search_dir} depending on the choice of iterates update. Then, for any $k \geq 0$ there exist constants $0 < \mu_1 \leq \mu_2$ such that
        $\mu_1 I \preceq A_k \preceq \mu_2 I$.
    \end{lemma}
    \vspace{-5pt}
    We introduce assumptions and theoretical results for the strongly convex case. 
    \begin{assumption} \label{assum:strong_conv} 
        There exist positive constants $\mu$ and $L$ such that $\mu I \preceq \nabla^2F(w) \preceq L I,$ $\forall w \in \mathbb{R}^d$.
    \end{assumption}
    \vspace{-5pt}
    The following theorem establishes global linear convergence of FLECS under strong convexity.
    \begin{theorem} \label{thm:main}
        Let Assumptions \ref{assum:diff} and \ref{assum:strong_conv} hold. Then for iterates $\{w_k\}$ generated by Algorithm~\ref{alg:main} with  $0 < \alpha_k =\alpha  \leq \frac{\mu_1}{\mu_2^2 L}$ we have  
        $F(w_k) - F^{\star}  \leq  ( 1-\alpha \mu \mu_1 )^k  [ F(w_0) - F^{\star}  ]
       ~ $ for all $k \geq 0$.
    \end{theorem}
    We are now ready to present a convergence guarantee which shows local linear convergence for FLECS under strong convexity. Note that, while Theorem~\ref{thm:main} shows that FLECS converges linearly under strong convexity, the rate depends on parameters related to the Hessian approximations. The following local rate is \emph{independent} of problem conditioning.
     \begin{theorem}\label{thm:local}
        Let Assumptions \ref{assum:diff} and \ref{assum:strong_conv} hold, and $\|w^0 - w^*\| \leq \frac{\mu^2}{2}$, $\frac{1}{n}\sum \limits_{i=1}^n \|B_{k+1}^i - \nabla^2 f_i(w^*)\|_\text{F} \leq \frac{2\mu^2}{L^2}$. Then for iterates $\{w_k\}$ generated by Algorithm~\ref{alg:main} with Truncated inverse Hessian approximation step (Algorithm \ref{alg:approx_trunc}) with parameters $\alpha_k =\alpha=1, ~0 \leq \omega \leq \mu, ~ \Omega \geq L$ we have $$\|w_k - w^*\|^2 \leq \tfrac{1}{2^k}\|w_0 - w^*\|^2$$ for all $k \geq 0$. 
    \end{theorem}
    \vspace{-5pt}
    We introduce the assumptions for nonconvex case.
    \vspace{-5pt}
    \begin{assumption}\label{assum:boundness}
        The function $F$ is bounded below by a scalar $\hat{F}$.
    \end{assumption}
    \vspace{-5pt}
    \begin{assumption}\label{assum:lipschits_grads}
        The gradients of $F$ are $L$-Lipschitz continuous for all $w \in \R^n$.
    \end{assumption}
    \vspace{-5pt}
    The following result shows that FLECS is convergent when problem \eqref{eq:problem} is nonconvex.
    \begin{theorem}\label{thm:nonconv}
        Let Assumptions\ref{assum:diff}, \ref{assum:boundness}, \ref{assum:lipschits_grads} hold and $w_0$ be the starting point. Then after $T > 0$ iterations of Algorithm~\ref{alg:main} with $0 < \alpha_k = \alpha  \leq \frac{\mu_1}{\mu_2^2 L}$,  we have 
        \begin{equation*}
            \tfrac{1}{T} \textstyle{\sum}_{k=0}^{T-1} \|\nabla F(w_k)\|^2 \leq \tfrac{2[F(w_0) - \hat{F}]}{\alpha \mu_1 T} \stackrel{T\to \infty}{\longrightarrow}0.
        \end{equation*}
    \end{theorem}
    
\section{Numerical Experiments}\label{sec:experiments}
\begin{figure}
\includegraphics[width=0.24\textwidth]{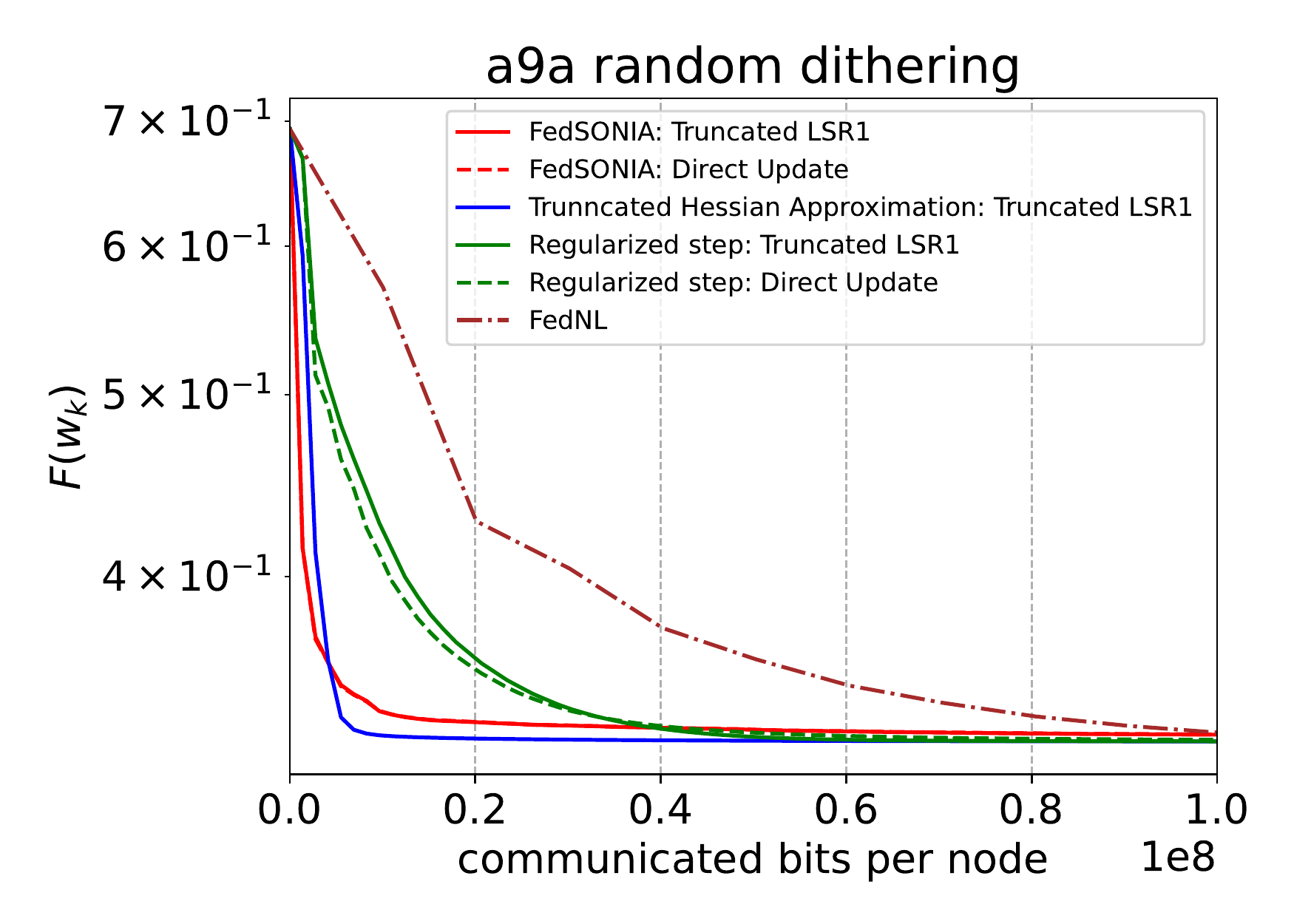}
\includegraphics[width=0.24\textwidth]{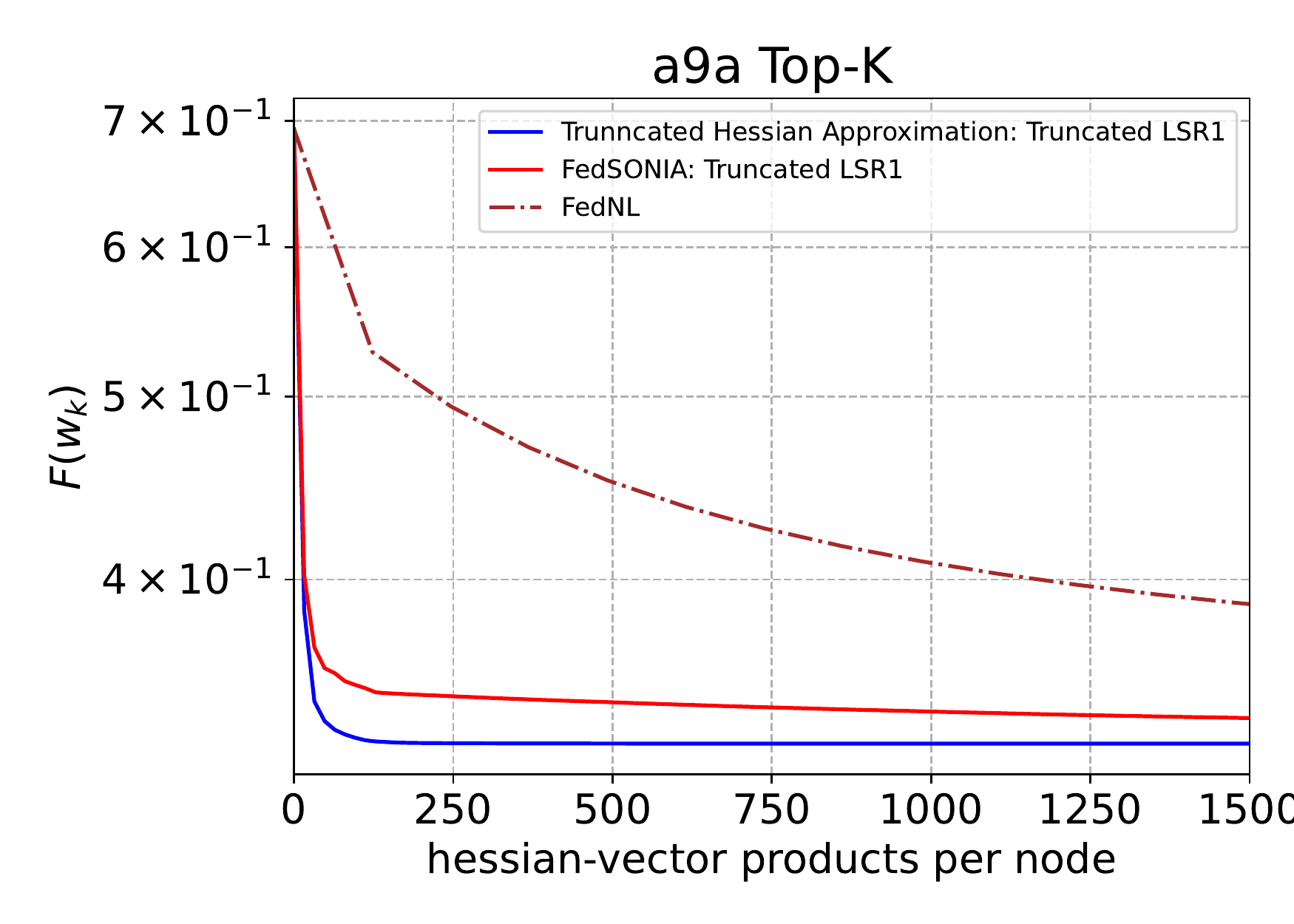}
\includegraphics[width=0.24\textwidth]{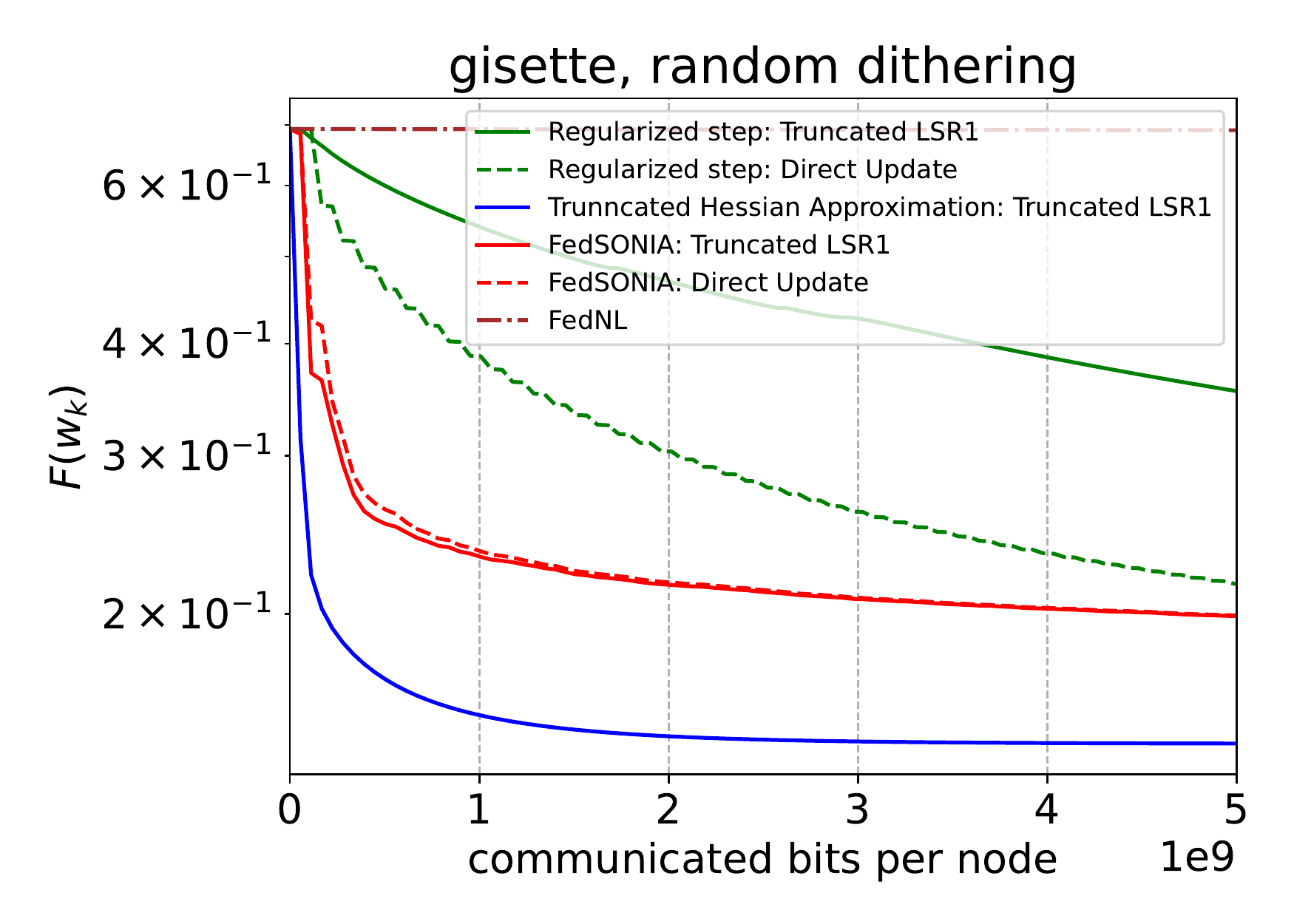}
\includegraphics[width=0.24\textwidth]{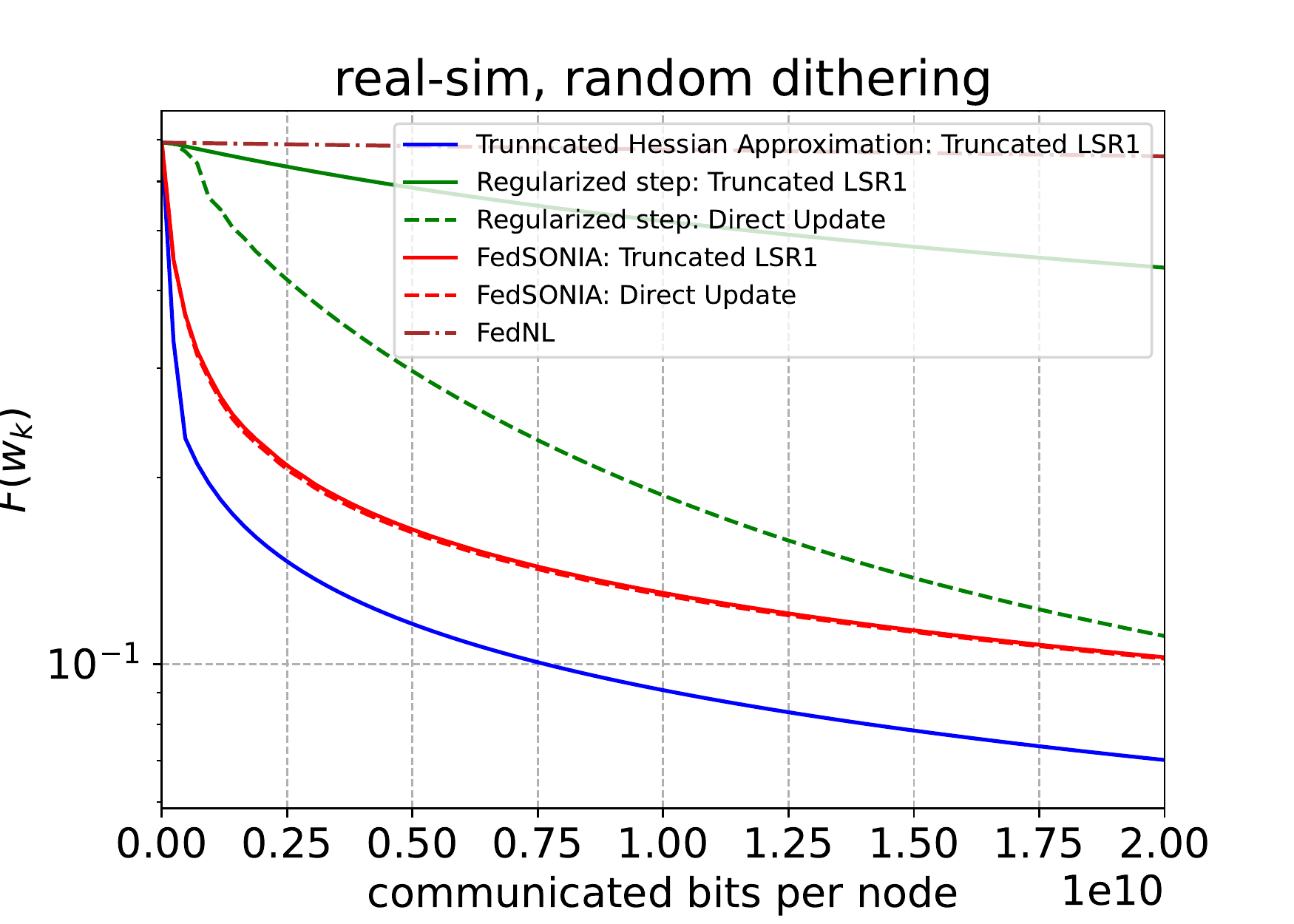}
\\
\includegraphics[width=0.24\textwidth]{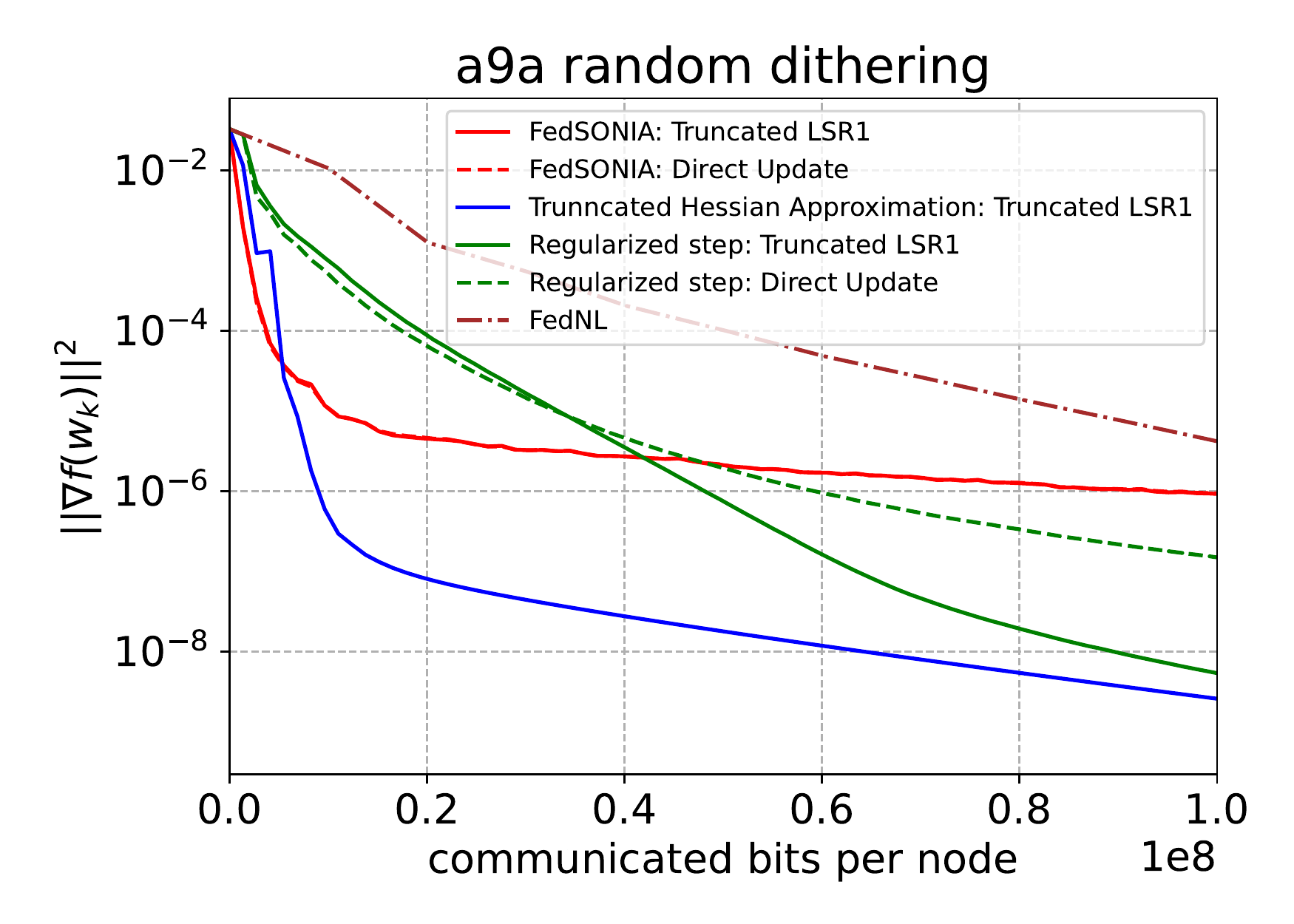}
\includegraphics[width=0.24\textwidth]{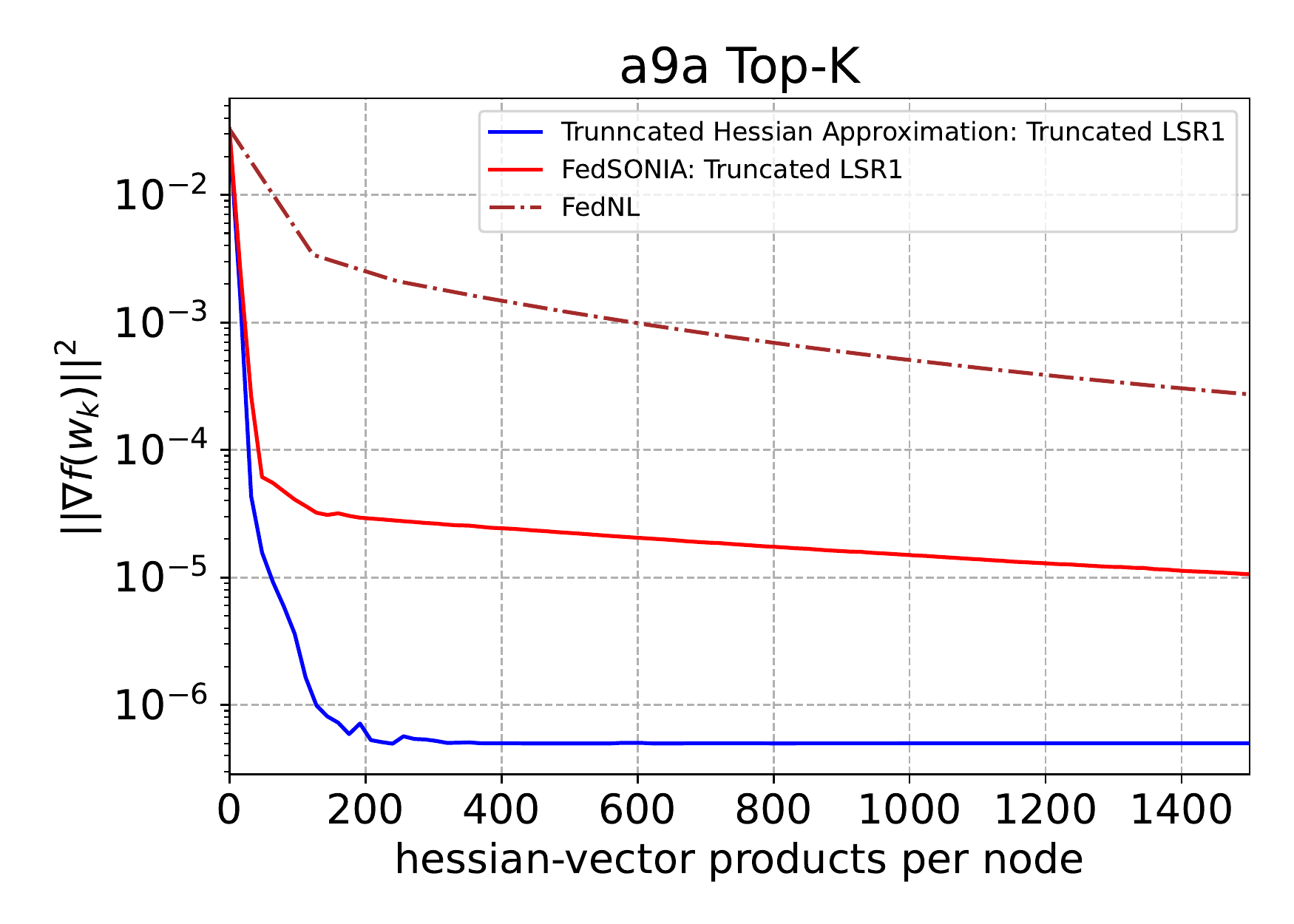}
\includegraphics[width=0.24\textwidth]{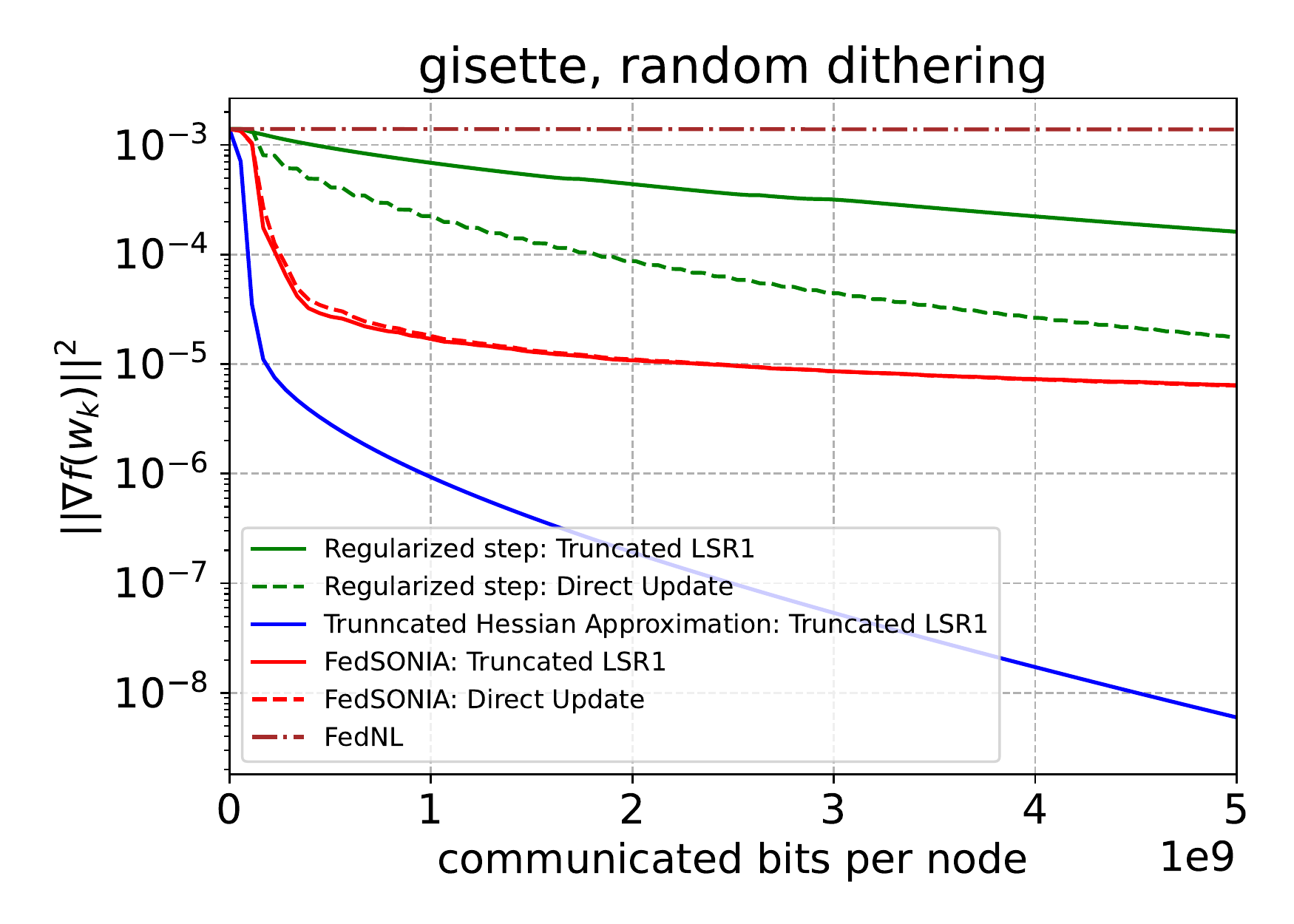}
\includegraphics[width=0.24\textwidth]{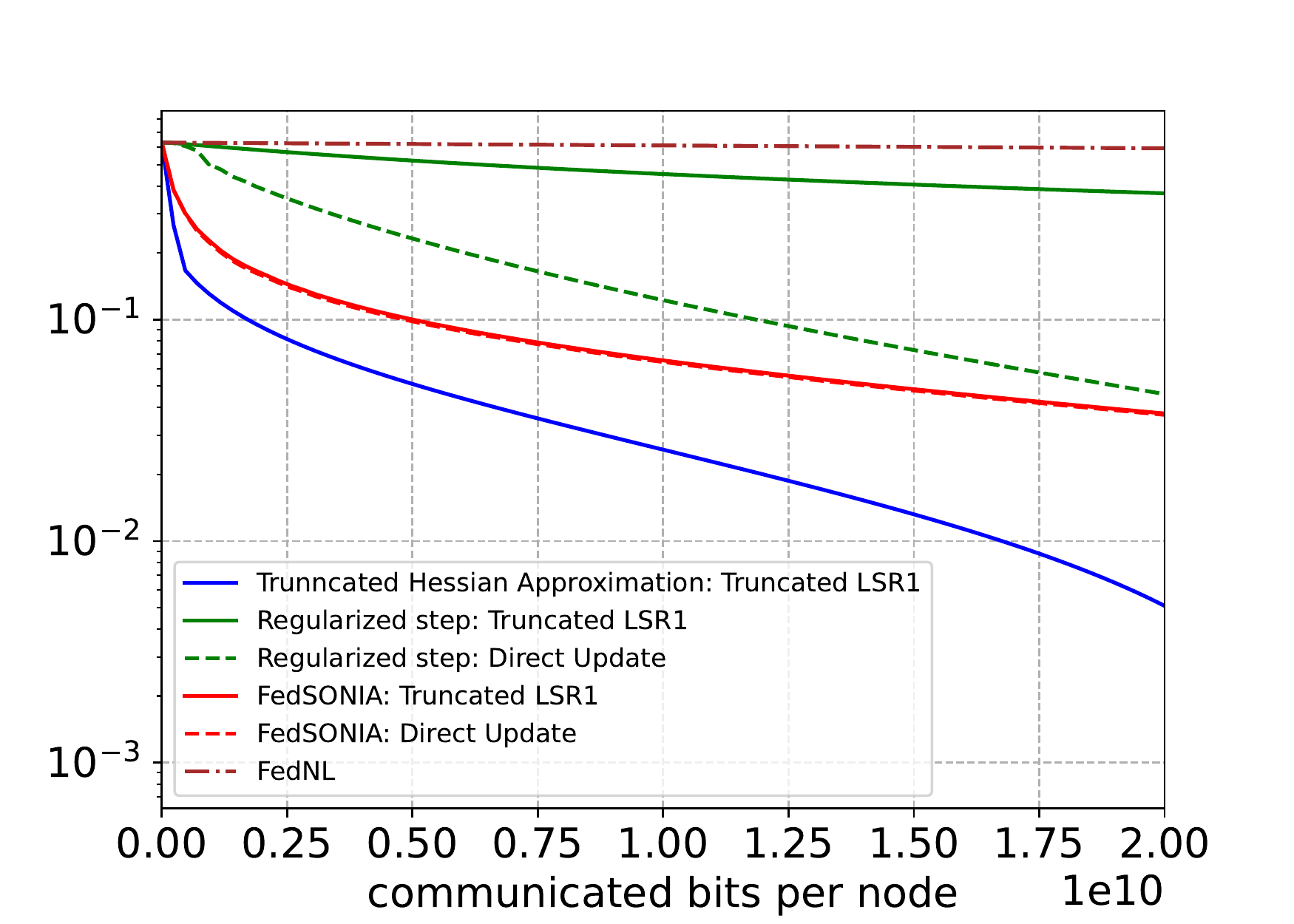}
\caption{Comparison of objective function $F(w_k)$ and the squared norm of gradient $\|\nabla F(w_k)\|^2$.}
\vskip-10pt
\end{figure}
We present numerical experiments for FLECS on a regularized logistic regression problem:
\begin{equation*}
    \min \lb \tfrac{1}{n} \textstyle{\sum}_{i=1}^n \tfrac{1}{r} \textstyle{\sum}_{j=1}^r \log (1+ \exp(-b_{ij}a_{ij}^Tw)) + \mu \|w\|^2\rb,
\end{equation*}
where $\lb a_{ij}, b_{ij}\rb_{j \in [m]}$ are data points on $i$-th device. We use three datasets from the LIBSVM library \cite{chang2011libsvm}: a9a ($123$ features), gisette-scale ($5000$ features) and real-sim ($20958$ features). We use Top-$K$ with $k=4d$ and random dithering \cite{alistarh2017qsgd} applied to each matrix column  with $s=128$ levels and $\infty$-norm as a matrix compressors. The performance of FLECS is compared with FedNL. \\[2pt]
In our experiments we set regularization parameter $\mu=10^{-5}$, the memory size to $m=16$ and use $B_k^i = 0$ as the initial Hessian approximation for FLECS and FedNL for random dithering  and $B_k^i = \nabla f_i(w_0)$ for Top-$K$. Truncation parameters were set to $\omega = 10^{-3}, ~ \Omega=10^8$, and for the FedSONIA update we set $\rho = \frac{1}{\Omega}$ and $\alpha = 1$.
For FLECS we use a Truncated L-SR1 Update (Algorithm \ref{alg:hess_lsr1}) and a Direct update (Algorithm \ref{alg:hess_direct}) with learning rate $\beta = 1$ as the Hessian learning techniques. For FedNL we use the same $\beta$.
\\[2pt]
The following three rules were used to compute the search direction $p_k$: (1) a Truncated Inverse Hessian Approximation (Algorithm \ref{alg:approx_trunc}), (2) FedSONIA (Algorithm \ref{alg:approx_fedsonia}) with $\rho_k = \frac{1}{\Omega}$, and (3) a
 regularized update where $B_{k+1} = B_k + \frac{1}{n}\sum_{i=1}^n\|C_k^i\|I$, where $C_k^i$ are defined in \eqref{eq:node_results}. The last update is motivated by \cite{safaryan2021fednl}. Note, that in the case of FLECS this update does not guarantee positive definiteness of Hessian approximation $B_k$ if $m < d$, but it can still perform well if $m$ is large enough. 
 \\[2pt]
    For FedNL after random dithering the matrix $B_{k+1}$ might not be symmetric, so we use $\frac{1}{2}(B_{k+1} + B_{k+1}^T)$ as Hessian approximation. In this case, we compare algorithms on the number of bits sent to the server per iteration. For Top-K compression, the communication complexity is roughly the same, so we estimate the performance based on the number of Hessian-vector products.
 \\[2pt]
For the iterates update \eqref{eq:update_rule} we set the step-size $\alpha = 1$ for all algorithms and datasets except FedSONIA on gisette and real-sim, where we set $\alpha = 0.1$, as it gives the best performance. Further details specifying algorithm parameters and implementation, as well as additional numerical experiments, can be found in Appendix~\ref{app:numerical}.

\section{Conclusion}\label{sec:conclusion}

This work presents a new second order framework for Federated Learning problems. The framework employs sketching and compression to ensure communication costs are low, and Hessian approximations are stored on the central server so that memory requirements for the workers is low. Due to the inclusion of partial approximate second order information, the algorithm enjoys favourable convergence guarantees, and numerical experiments show the practical benefits of our approach.

\clearpage
\bibliography{bibliography}

\clearpage
\appendix

\section{Proofs}\label{app:proofs}
    \subsection{Assumptions}

    \begin{customassum}{\ref{assum:diff}}
        The function $F$ is twice continuously differentiable.
    \end{customassum}
    \begin{customassum}{\ref{assum:strong_conv}}
        There exist positive constants $\mu$ and $L$ such that $\mu I \preceq \nabla^2F(w) \preceq L I,$ $\forall w \in \mathbb{R}^d.$
    \end{customassum}
    \begin{customassum}{\ref{assum:boundness}}
        The function $F$ is bounded below by a scalar $\hat{F}$.
    \end{customassum}
    \begin{customassum}{\ref{assum:lipschits_grads}}
        The gradients of $F$ are $L$-Lipschitz continuous for all $w \in \R^n$.
    \end{customassum}
    
    \subsection{Proof of Lemma \ref{lem:positive_df}}
    \begin{customlemma}{\ref{lem:positive_df}}
        Let Assumption \ref{assum:diff} hold. Let $A_k$ be defined as in \eqref{eq:search_dir} depending on the choice of iterates update. Then, for any $k \geq 0$ there exist constants $0 < \mu_1 \leq \mu_2$ such that
        $$\mu_1 I \preceq A_k \preceq \mu_2 I.$$
    \end{customlemma}
    \begin{proof}
        \textit{Truncated L-SR1 update.} 
        In this case $A_k = (|B_{k+1}|_\omega^\Omega)^{-1}$. 
        By spectral decomposition $B_{k} = V_{k} \Lambda_{k} V_{k}^T$. By truncation (Definition \ref{def:truncation})
        \begin{equation*}
            \omega I \preceq |\Lambda_k|_{\omega}^\Omega
            \preceq \Omega I.
        \end{equation*}
        Furthermore, $\dfrac 1 \Omega I \preceq A_k  \preceq \dfrac{1}{\omega}I $. Defining $\mu_1 \eqdef \frac{1}{\Omega}$ and $\mu_2 \eqdef \frac{1}{\omega}$, and noting that $0< \mu_1 \leq \mu_2$, gives the result.
        
        \textit{FedSONIA update.}
        In this case $A_k = \left(|B_{k+1}^{\text{SONIA}}|_\omega^\Omega\right)^{-1} + \rho_k (I-\widetilde{V}_k\widetilde{V}_k^T)$. By \eqref{eq:A_sonia} 
        $$A_k = \V_k(|\Lambda_k|_\omega^\Omega)^{-1}\V_k^T = \V_k (|\Lambda_k|_\omega^\Omega)^{-1} - \rho_k I)\V_k^T + \rho_kI \stackrel{\eqref{eq:search_direction_sonia}}{\succeq} \rho_k I.$$
        By truncation (Definition \ref{def:truncation})
        $\omega I \preceq |\Lambda_k|_{\omega}^\Omega\preceq \Omega I.$
        Therefore, $A_k \preceq \frac{1}{\omega} I + \rho_k I \preceq \frac{2}{\omega}$. Defining $\mu_1 \eqdef \frac{1}{\Omega}$ and $\mu_2 \eqdef \frac{2}{\omega}$, and noting that $0< \mu_1 \leq \mu_2$, gives the result.
        
    \end{proof}
    
    \subsection{Proof of Theorem \ref{thm:main}}
    
    \begin{customthm}{\ref{thm:main}}
        Let Assumptions \ref{assum:diff}, and \ref{assum:strong_conv} hold. Then for iterates $\{w_k\}$ generated by Algorithm~\ref{alg:main} with  $0 < \alpha_k =\alpha  \leq \frac{\mu_1}{\mu_2^2 L}$ we have 
        $$F(w_k) - F^{\star}  \leq  ( 1-\alpha \mu \mu_1 )^k  [ F(w_0) - F^{\star}  ] $$ 
        for all $k \geq 0$.
    \end{customthm}
    \begin{proof} We have 
    \begin{gather} 
        F(w_{k+1}) = F(w_k -\alpha A_k \nabla F(w_k)) \nonumber \\
        \stackrel{\text{Assump. \ref{assum:strong_conv}}}{\leq} F(w_k) + \nabla F(w_k)^T (-\alpha A_k \nabla F(w_k)) + \frac{L}{2}\| \alpha A_k \nabla F(w_k)\|^2 \nonumber \\
        \stackrel{Lem. \ref{lem:positive_df}}{\leq} F(w_k) - \alpha \mu_1 \| \nabla F(w_k) \|^2  + \frac{\alpha^2 \mu_2^2 L}{2} \| \nabla F(w_k)\|^2 \nonumber \\
         = F(w_k) - \alpha \left(\mu_1 - \alpha \frac{\mu_2^2 L}{2} \right)\| \nabla F(w_k) \|^2 
        {\leq} F(w_k) - \alpha \frac{\mu_1}{2} \| \nabla F(w_k) \|^2,  \label{eq:proof_step_app}
    \end{gather}
    where the last inequality is due to the choice of the step-size. By Assumtion \ref{assum:strong_conv}, we have $2\mu(F(w)-F^\star) \leq \| \nabla F(w)\|^2$. Therefore,
    \begin{align*}
        F(w_{k+1}) \leq F(w_k) - \alpha \mu \mu_1 (F(w_k)-F^\star).
    \end{align*}
    Subtracting $F^\star$ from both sides, 
    \begin{align*}
        F(w_{k+1}) - F^\star \leq (1 - \alpha \mu \mu_1) (F(w_k)-F^\star).
    \end{align*}
    Recursive application of the above inequality yields the desired result.
    \end{proof}
     \subsection{Proof of Theorem \ref{thm:local}}\label{app:local}
    We begin with the technical lemma for FLECS with Truncated Hessian inverse Approximation step (Algorithm \ref{alg:approx_trunc}).
    \begin{lemma}\label{lem:trunc}
        Let Assumptions \ref{assum:diff}, \ref{assum:strong_conv} hold and $0 \leq \omega \leq \mu$, $\Omega \geq L$. Then for Hessian approximations $B_{k+1}$ we have 
        \begin{equation*}
            \| |B_{k+1}|_\omega^\Omega - \nabla^2 f(w^*)\|_\text{F} \leq\| B_{k+1} - \nabla^2 f(w^*)\|_\text{F}
        \end{equation*}
        for all $k \geq 0$.
    \end{lemma}    
\begin{proof}
    By Definition \ref{def:truncation}
    $$|B_{k+1}|_\omega^\Omega\ := V_{k+1} |\Lambda_k|_\omega^\Omega V_{k+1}^T,$$
    where $B_{k+1} = V_{k+1}\Lambda_k V_{k+1}^T$ is spectral decomposition of $B_{k+1}$. Then,
    \begin{gather*}
        \left\| |B_{k+1}|_\omega^\Omega - \nabla^2 f(w^*)\right\|_\text{F} = \left\| V_{k+1} |\Lambda_{k+1}|_\omega^\Omega V_{k+1}^T - \nabla^2 f(w^*)\right\|_\text{F} \\
        \stackrel{V_{k+1}^T V_{k+1} = I}{=} \left\| V_{k+1} |\Lambda_{k+1}|_\omega^\Omega V_{k+1}^T  - V_{k+1} V_{k+1}^T \nabla^2 f(w^*) V_{k+1}V_{k+1}^T\right\|_\text{F} \\
        = \left\|  |\Lambda_{k+1}|_\omega^\Omega   -  V_{k+1}^T \nabla^2 f(w^*) V_{k+1}\right\|_\text{F}.
    \end{gather*}
    Analogically,
    $$\left\|B_{k+1} - \nabla^2 f(w^*)\right\|_\text{F} = \left\|  \Lambda_k  -  V_{k+1}^T \nabla^2 f(w^*) V_{k+1}\right\|_\text{F}.$$
    Now, let us denote $\lambda_i, ~i=1, \ldots, d$ as eigenvalues of $B_{k+1}$ (diagonal elements of $\Lambda_{k+1}$). Let  $\wlambda_1, \wlambda_2, \ldots, \wlambda_n$ be corresponding diagonal elements of $|\Lambda_{k+1}|_\omega^\Omega$. 
    
    Noting, that $$0 \preceq \omega I \preceq \mu I \preceq \nabla^2 F(w^*) \preceq LI \preceq \Omega I,$$ we have $|\wlambda_i - h_i| \leq |\lambda_i - h_i|$ for all $i = 1, \ldots, n$, where $h_i = \left[V_{k+1}^T \nabla^2 F(w^*) V_{k+1} \right]_{ii}$. 
    
    Therefore, 
    \begin{equation*}
            \| |B_{k+1}|_\omega^\Omega - \nabla^2 f(w^*)\|_\text{F} \leq\| B_{k+1} - \nabla^2 f(w^*)\|_\text{F}
    \end{equation*}
    by definition of Frobenius norm and the fact that matrices $|B_{k+1}|_\omega^\Omega - \nabla^2 f(w^*)$ and \\$\Lambda_k  -  V_{k+1}^T \nabla^2 f(w^*) V_{k+1}$ differ only in diagonal elements. 
\end{proof}

    \begin{customthm}{\ref{thm:local}}
        Let Assumptions \ref{assum:diff} and \ref{assum:strong_conv} hold, and $\|w^0 - w^*\| \leq \frac{\mu^2}{2}$, $\frac{1}{n}\sum \limits_{i=1}^n \|B_{k+1}^i - \nabla^2 f_i(w^*)\|_\text{F} \leq \frac{2\mu^2}{L^2}$. Then for iterates $\{w_k\}$ generated by Algorithm~\ref{alg:main} with Truncated inverse Hessian approximation step (Algorithm \ref{alg:approx_trunc}) with parameters $\alpha_k =\alpha=1, ~0 \leq \omega \leq \mu, ~ \Omega \geq L$ we have $$\|w_k - w^*\|^2 \leq \tfrac{1}{2^k}\|w_0 - w^*\|^2$$ for all $k \geq 0$. 
    \end{customthm}
    \begin{proof}
        \begin{gather*}
            \|w_{k+1} - w^*\|^2 = \|w_k - w^* - A_k \nabla F(w_k)\|^2 
            \leq \|A_k\|^2\|A_k^{-1}(w_k - w^*) - \nabla F(w_k)\|^2 \\
            \stackrel{\text{Lem. \ref{lem:positive_df}}}{\leq} \mu_2^2 \|A_k^{-1}(w_k - w^*) - \nabla F(w_k)\|^2 \\
            \leq 2\mu_2^2 \left(\| \left(A_k^{-1} - \nabla^2 F(w^*)\right)(w_k - w^*)\|^2  +  \| \nabla^2 F(w^*)(w_k - w^*) - \nabla F(w_k) + \nabla F(w^*)\|^2\right) \\
            \stackrel{\text{Assump. \ref{assum:strong_conv}}}{\leq} 2\mu_2^2 \left(\| A_k^{-1} - \nabla^2 F(w^*)\|^2 \|w_k - w^*\|^2 +  \frac{L^2}{4}\|w_k - w^*\|^4\right)\\
            \leq 2\mu_2^2\|w_k - w^*\|^2 \left(\| A_k^{-1} - \nabla^2 F(w^*)\|_\text{F}^2  +  \frac{L^2}{4}\|w_k - w^*\|_\text{F}^2\right)\\
            \stackrel{\eqref{eq:search_dir_sr1}}{=} \leq 2\mu_2^2\|w_k - w^*\|^2 \left(\| |B_{k+1}|_\omega^\Omega - \nabla^2 F(w^*)\|_\text{F}^2  +  \frac{L^2}{4}\|w_k - w^*\|_\text{F}^2\right)\\
            \stackrel{\text{Lem.\ref{lem:trunc}}}{\leq }2\mu_2^2\|w_k - w^*\|^2 \left(\| B_{k+1} - \nabla^2 F(w^*)\|_{F}^2  +  \frac{L^2}{4}\|w_k - w^*\|^2\right)
        \end{gather*}
         
        Therefore, by assumptions of the theorem
        \begin{gather*}
            \|w_{k+1} - w^*\|^2 \leq \frac{1}{\mu^2}\|w_k - w^*\|^2 \left(\|B_{k+1} - \nabla^2 F(w^*)\|_\text{F}^2  +  \frac{L^2}{4}\|w_k - w^*\|^2\right) \leq \frac{1}{2}\|w_k - w^*\|^2
        \end{gather*}
    \end{proof}
    \subsection{Proof of Theorem \ref{thm:nonconv}}\label{app:nonconv}

    \begin{customthm}{\ref{thm:nonconv}}
        Let Assumptions\ref{assum:diff}, \ref{assum:boundness}, \ref{assum:lipschits_grads} hold and $w_0$ be the starting point. Then after $T > 0$ iterations of Algorithm~\ref{alg:main} with $0 < \alpha_k = \alpha  \leq \frac{\mu_1}{\mu_2^2 L}$,  we have 
        \begin{equation*}
            \tfrac{1}{T} \textstyle{\sum}_{k=0}^{T-1} \|\nabla F(w_k)\|^2 \leq \tfrac{2[F(w_0) - \hat{F}]}{\alpha \mu_1 T} \stackrel{T\to \infty}{\longrightarrow}0.
        \end{equation*}
     \end{customthm}
    \begin{proof} By \eqref{eq:proof_step_app} 
    \begin{align*} 
        F(w_{k+1})  \leq F(w_k) - \alpha \frac{\mu_1}{2} \| \nabla F(w_k) \|^2.  
    \end{align*}
    Summing both sides of the above inequality from $k=0$  to $T -1$, we obtain
    \begin{align*} 
    	\widehat{F} -F(w_0) \stackrel{\text{Assum. \ref{assum:boundness}}}{\leq} F(w_{T})-F(w_0) = \sum_{k=0}^{T-1}(F(w_{k+1}) - F(w_k)) \leq  - \sum_{k=0}^{T-1}\alpha \frac{\mu_1}{2} \| \nabla F(w_k) \|^2. 
    \end{align*}
    Therefore, we have
    \begin{align} \label{eq:int_nonconvex}
        \sum_{k=0}^{T-1} \| \nabla F(w_k) \|^2  &  \leq \frac{2[ F(w_0) - \widehat{F}]}{\alpha \mu_1}.
    \end{align}
    Dividing \eqref{eq:int_nonconvex} by $T$ we conclude
    \begin{align*}
        \frac{1}{T}\sum_{k=0}^{T-1} \| \nabla F(w_k) \|^2  &  \leq \frac{2[ F(w_0) - \widehat{F}]}{\alpha \mu_1 T}.
    \end{align*}
    \end{proof}
   
\section{Additional Numerical Experiments}\label{app:numerical}

\subsection{Top-$K$ Compressor}\label{app:topk}

In this experiment we used Top-$K$ compression with $k=4d$ on a9a dataset with regularization parameter $\mu = 10^{-5}$. We compare FedNL \cite{safaryan2021fednl} and FLECS in the number of hessian-vector products, since compression cost is roughly the same. For FLECS we set memory size $m=16$ and step-size $\alpha=1$. For both algorithms we set $B_0^i = \nabla F(w_0)$ and Hessian approximation rate $\beta = 1$. In this experiment we omit initialization cost (computation of the full Hessian at starting point) for both algorithms.

\begin{figure}[h]
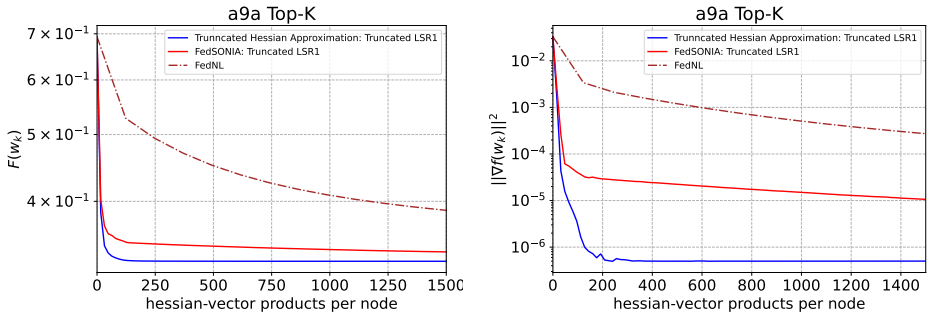

    \includegraphics[width=0.45\textwidth]{experiments/TopK/a9a_top_k_loss-3.pdf}
    \includegraphics[width=0.45\textwidth]{experiments/TopK/a9a_top_k_grad2-4.pdf}
    \vskip-10pt
    \caption{Comparison of objective function $F(w_k)$ and the squared norm of gradient $\|\nabla F(w_k)\|^2$ between FLECS and FedNL.}
\end{figure}

\subsection{Comparison with DIANA and ADIANA.}
In order to compare our framework with DIANA and ADIANA we use random dithering with number of levels $s=\sqrt{d}$ and $p=\infty$. For FLECS we set hyperparametes as: $m=1$, $\beta=1$, $\alpha=1$ and initialize Hessian approximations with $0$.

For DIANA and ADIANA we used theoretical parameters $L=1/4$ and strong convexity constant $\mu = 10^{-5}$. For these methods we used an original code provided by authors \cite{Mishchenko2019, Li2020accFL}.

\begin{figure}[h]
    \includegraphics[width=0.31\textwidth]{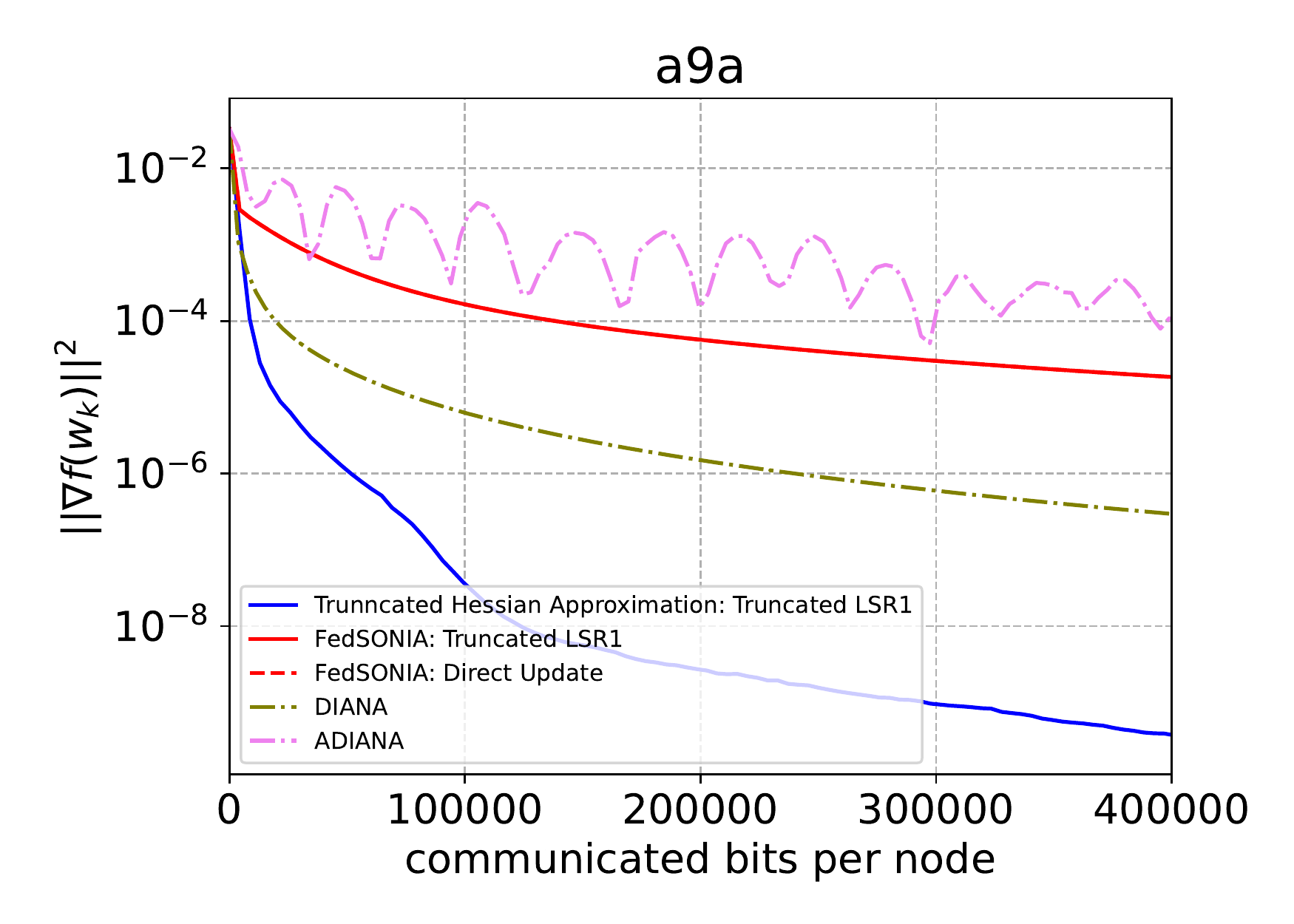}
    \includegraphics[width=0.31\textwidth]{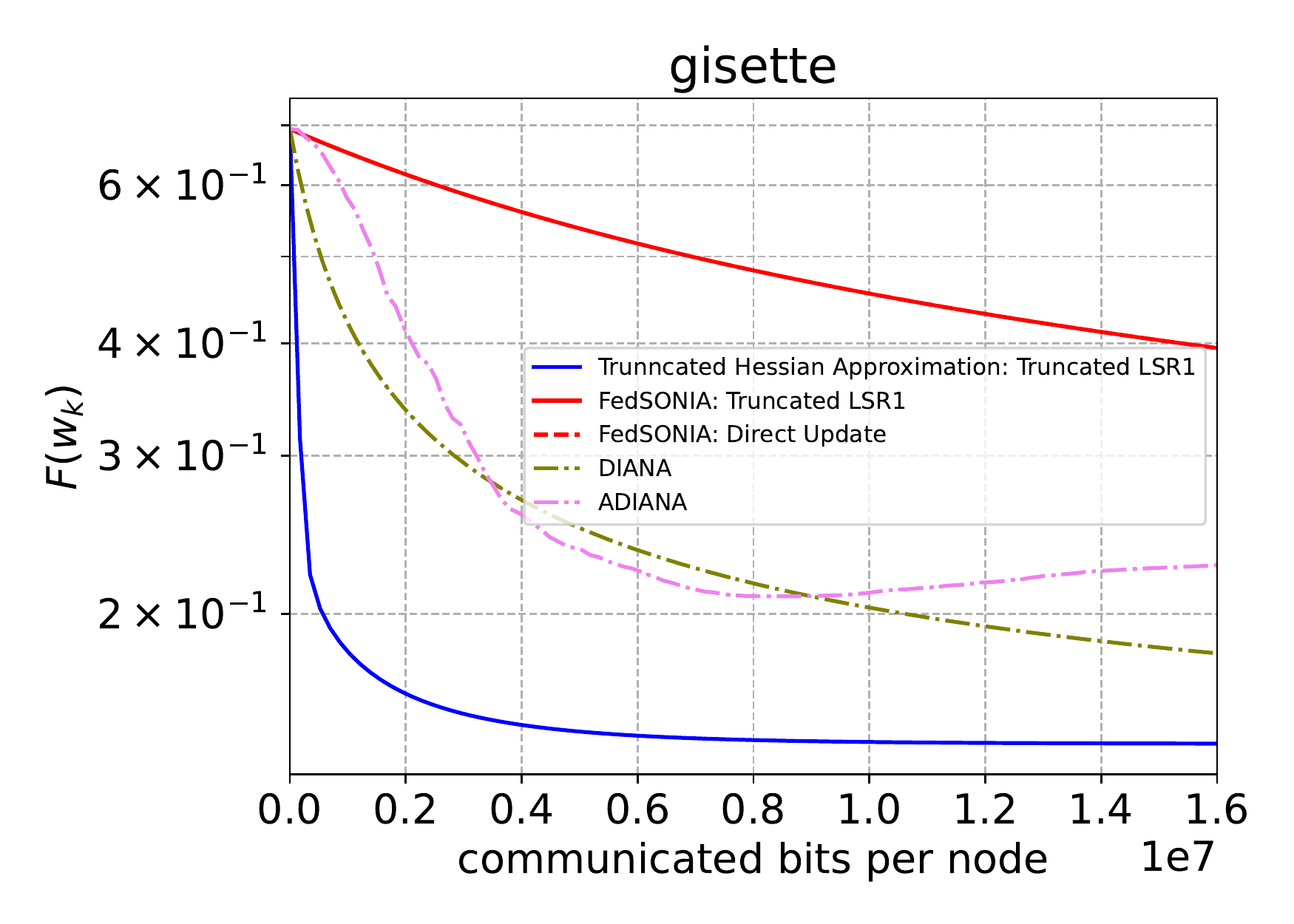}
    \includegraphics[width=0.31\textwidth]{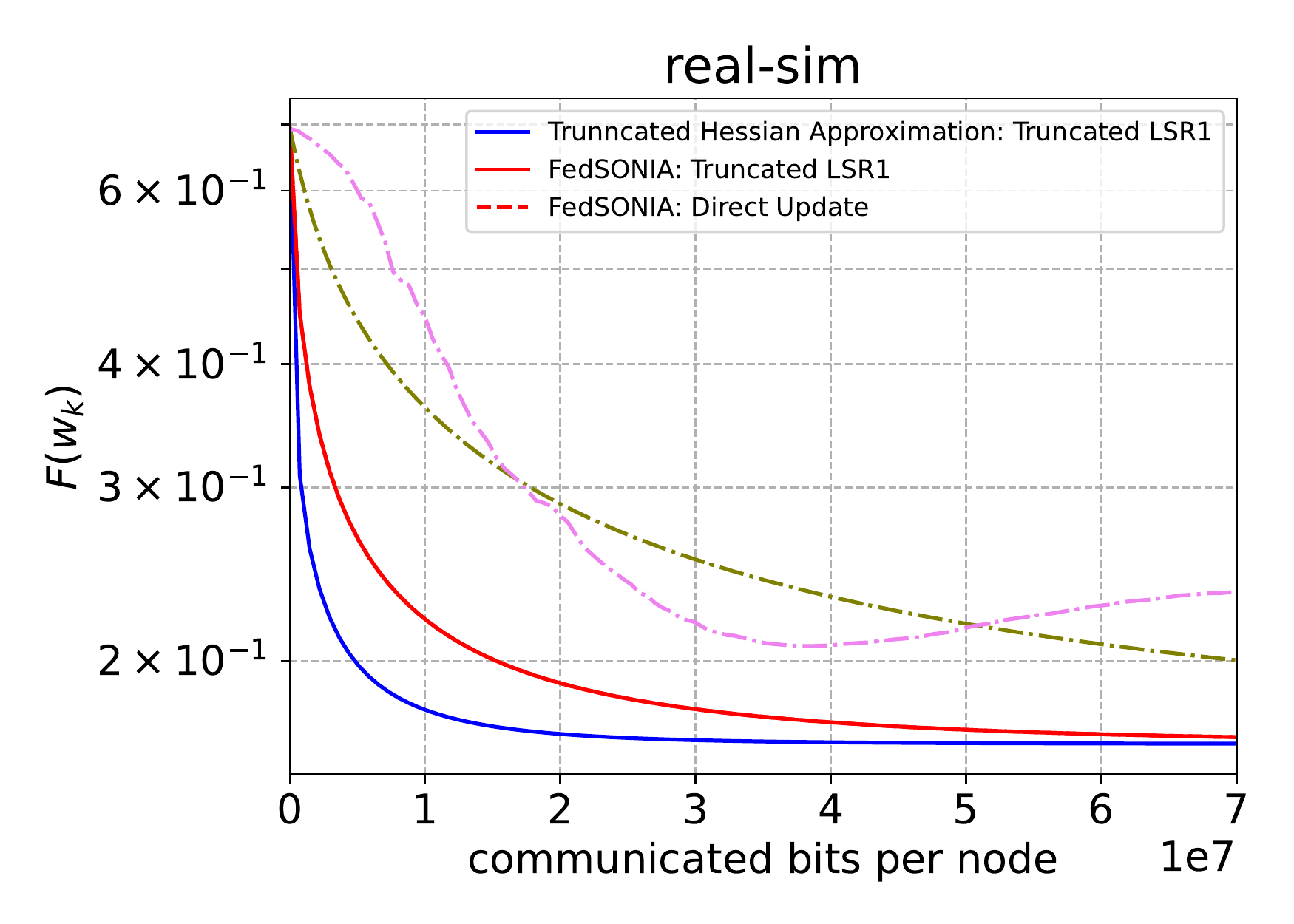}\\
    \includegraphics[width=0.31\textwidth]{experiments/mem1/a9a_1e-5_grad_diana-2.pdf}
    \includegraphics[width=0.31\textwidth]{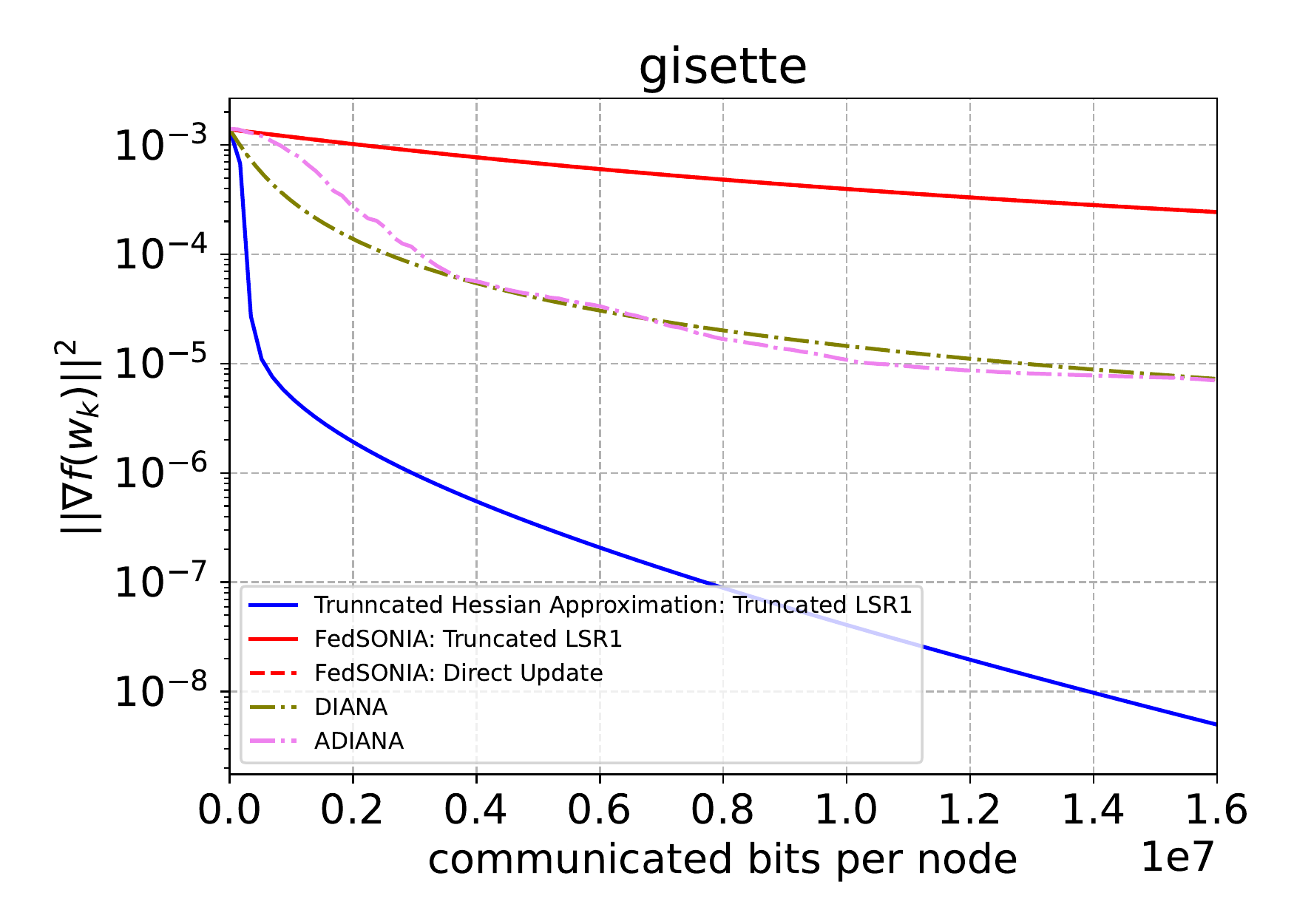}
    \includegraphics[width=0.31\textwidth]{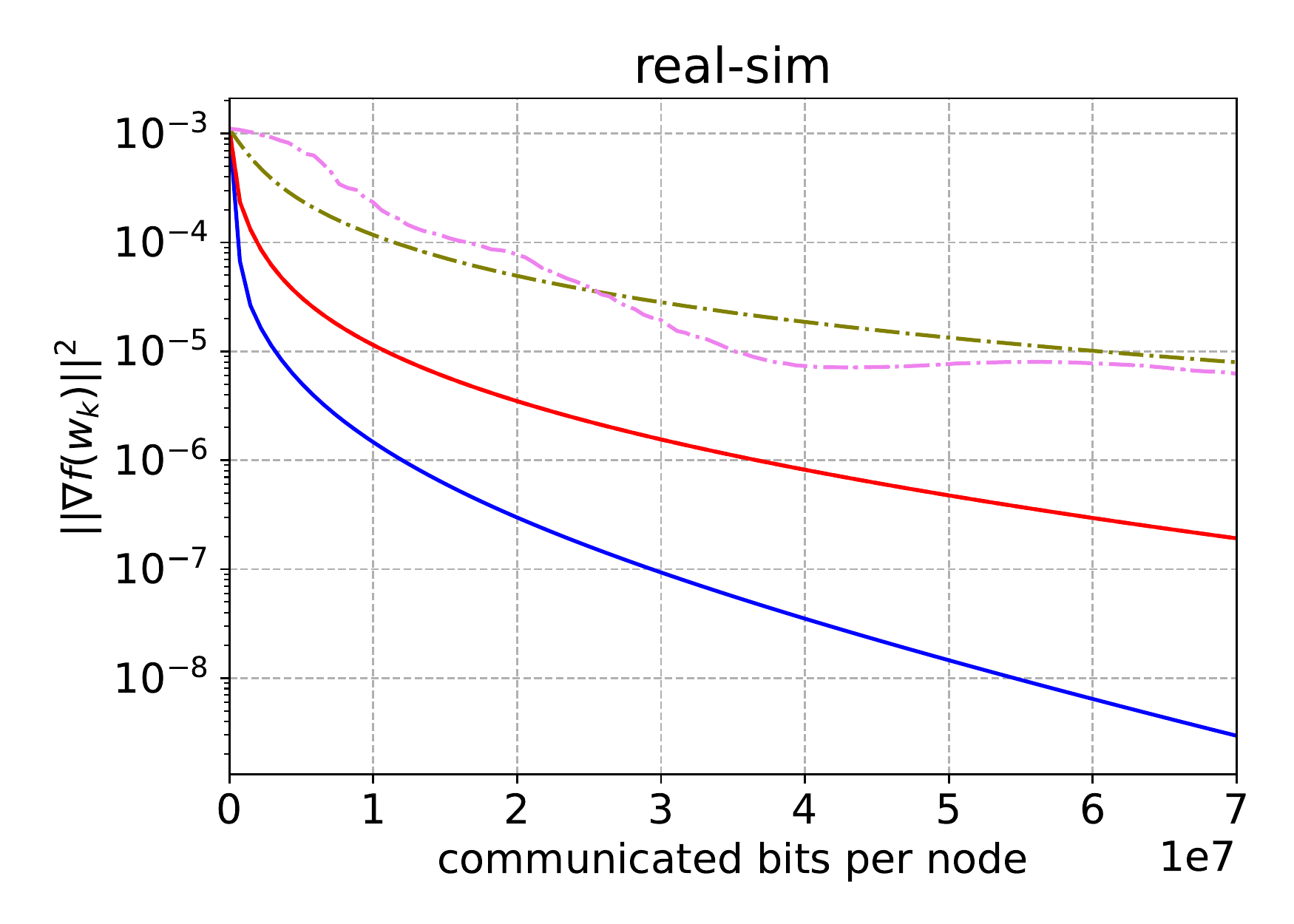}
    \vskip-10pt
    \caption{Comparison of objective function $F(w_k)$ and the squared norm of gradient $\|\nabla F(w_k)\|^2$ between DIANA, ADIANA and FLECS.}
\end{figure}

\subsection{Top-$K$ vs Random Dithering}
In this subsection we want to compare Top-$K$ compressor with random dithering for FLECS. We set hyperparametes as: $m=16$, $\beta=1$, $\alpha=1$ and initialize Hessian approximations with $\nabla^2 f(w_0)$. We use Top-$K$ with $k=4d$ and random dithering with
$s=128$ and $p=\infty$, so communication cost per iteration is roughly the same. In this experiment we omit initialization cost (computation of the full Hessian at starting point) for both algorithms. 
\begin{figure}[h]
    \includegraphics[width=0.45\textwidth]{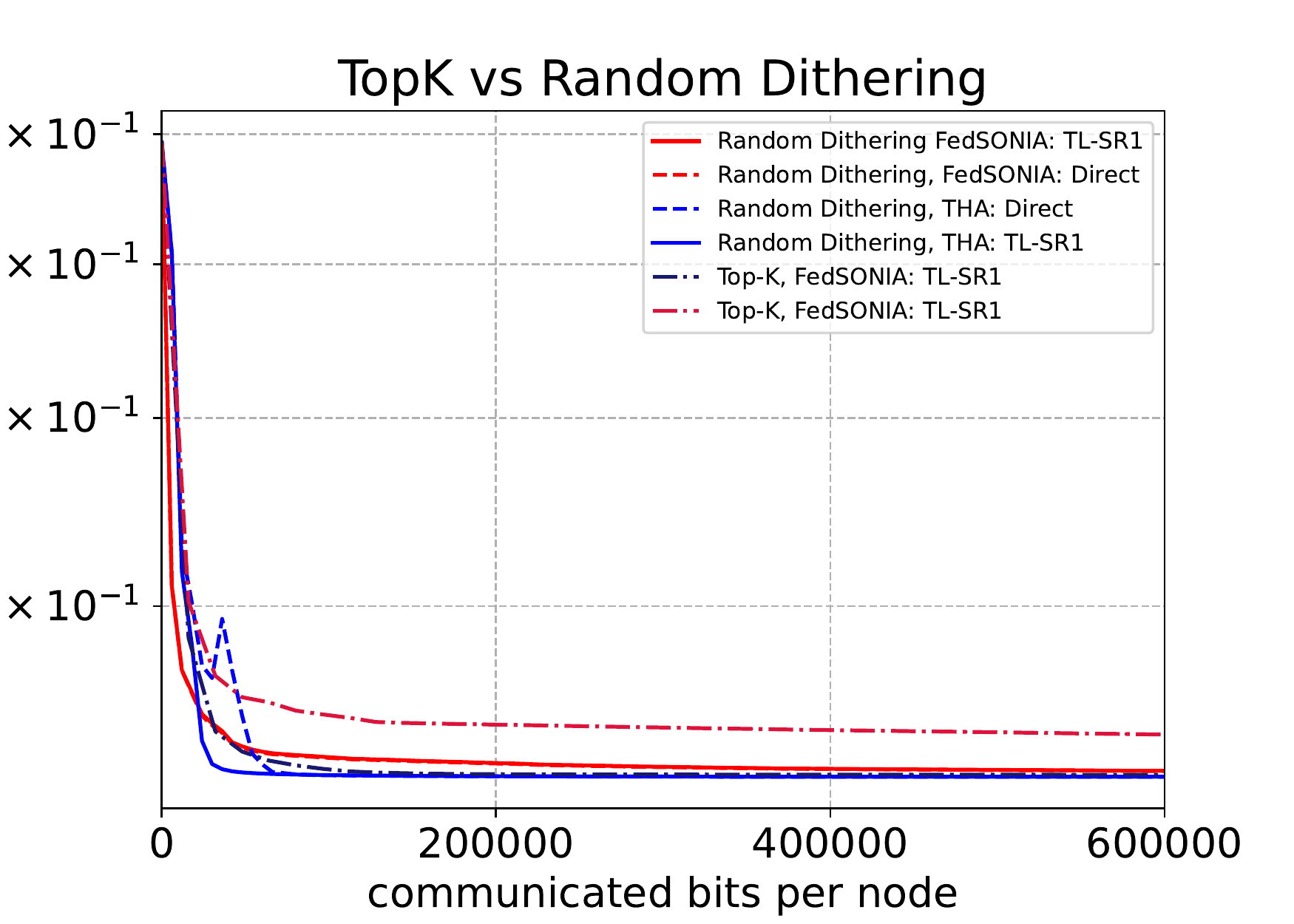}
    \includegraphics[width=0.45\textwidth]{experiments/took_vs_rd/topkvsrd_loss.pdf}
    \vskip-10pt
    \caption{Comparison of objective function $F(w_k)$ and the squared norm of gradient $\|\nabla F(w_k)\|^2$ .}
\end{figure}

\subsection{Dependence on the Memory Size}
Here we compare different memory sizes for FLECS on a9a dataset. We use random dithering with parameters $s=128$ and $p=\infty$. Hyperparametes set as $\beta=1$, $\alpha=1$ and initialize Hessian approximations initialized with $0$.
\begin{figure}[h]
    \includegraphics[width=0.45\textwidth]{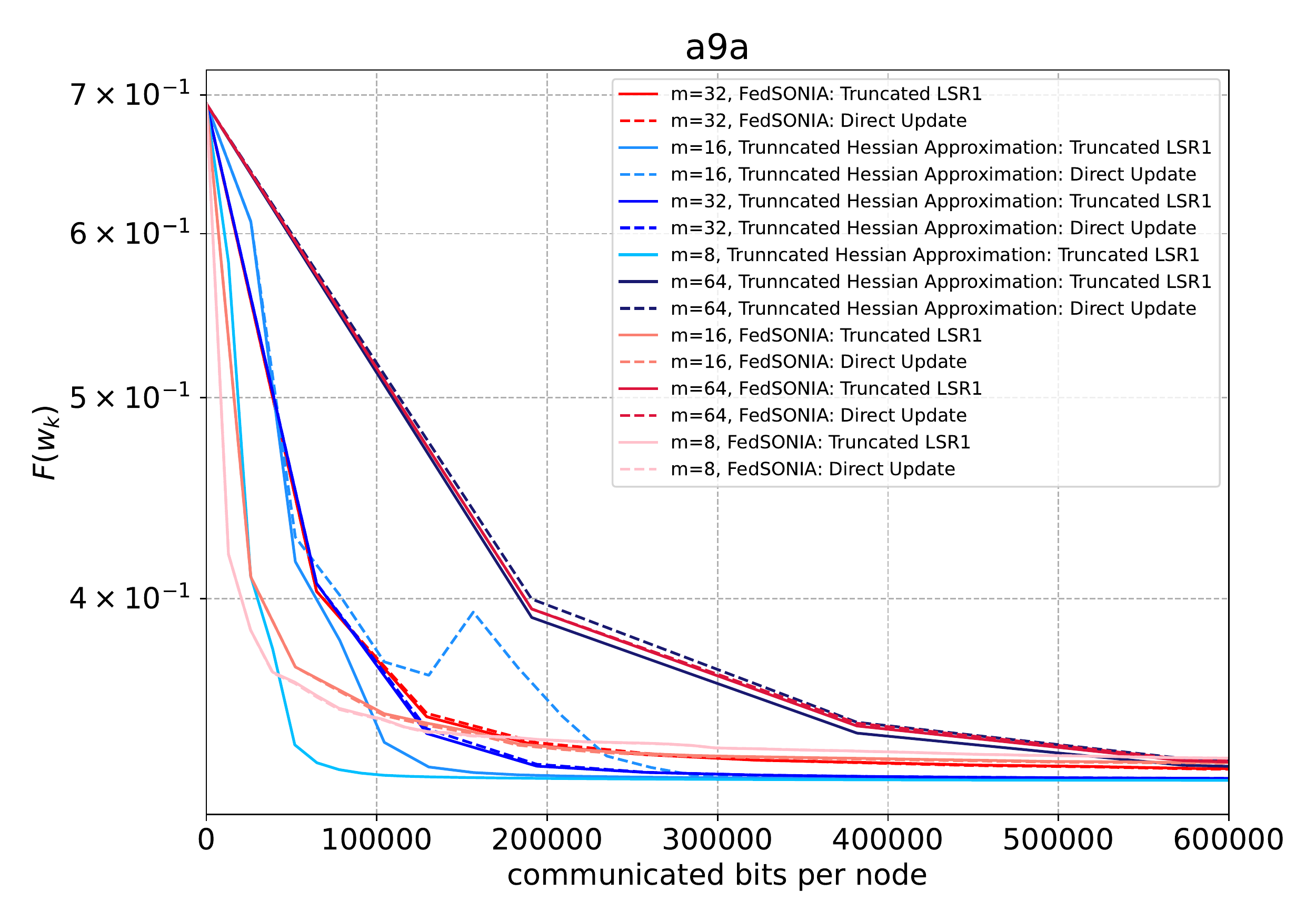}
    \includegraphics[width=0.45\textwidth]{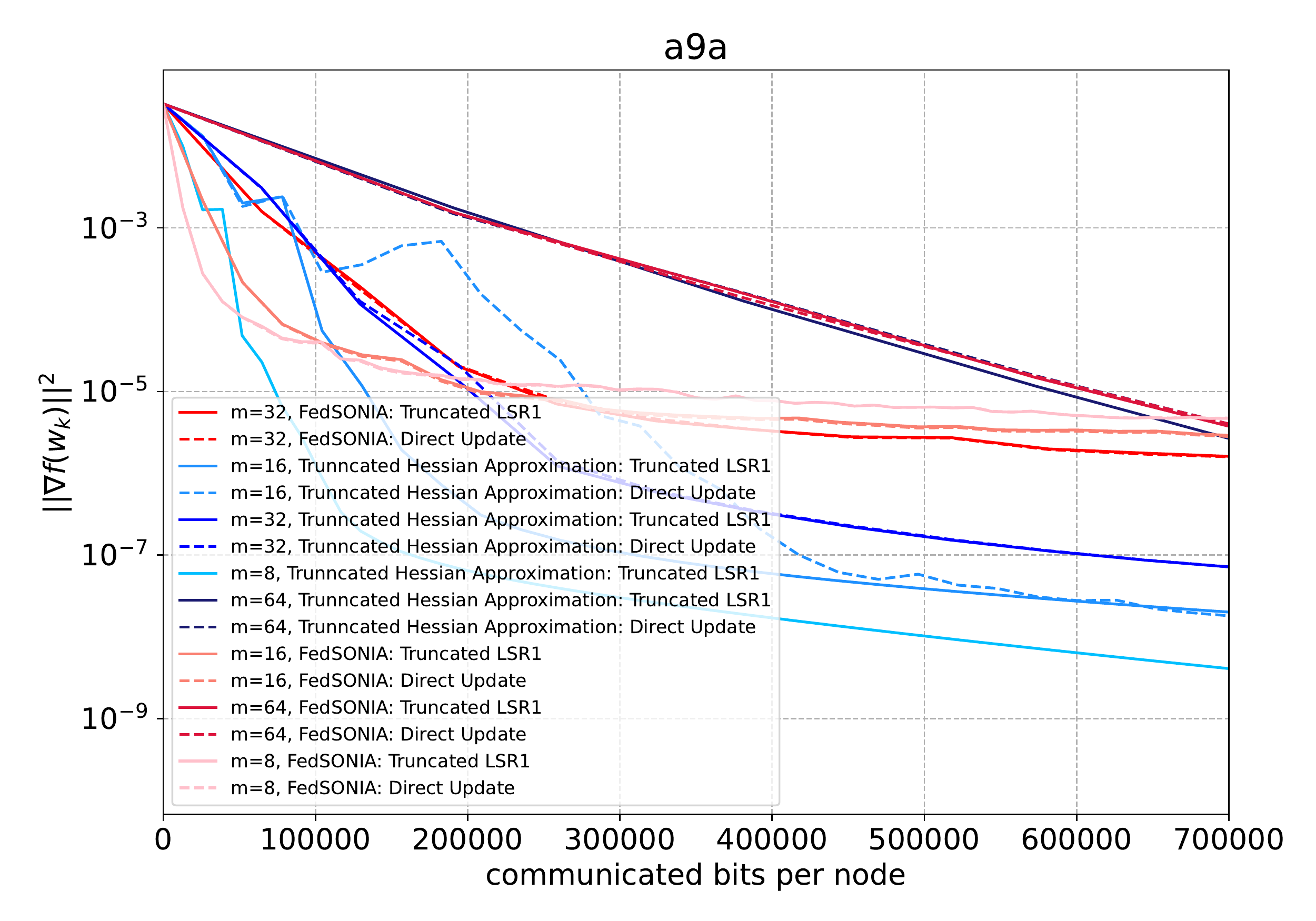}
    \vskip-10pt
    \caption{Comparison of objective function $F(w_k)$ and the squared norm of gradient $\|\nabla F(w_k)\|^2$ for different memory sizes.}
\end{figure}

\section{Compressors}\label{app:compressors}
FLECS allows several different compression operators. In this section we give definitions and provide examples of possible compressors for our framework.

\subsection{Unbiased Compressors.}
\begin{definition}
    A random operator $\C:~\R^{d\times m} \to \R^{d\times m}$ satisfying
    \begin{equation}\label{eq:quantization_def}
        \EE{\mathcQ}{\C(X)} = X, ~~ \EE{\C}{\|\C(X)\|_{\text{F}}^2} \leq (\omega + 1)\|X\|_{\text{F}}^2],
    \end{equation}
    for all $X \in \R^{d\times m}$ is a unbiased compressor operator.
\end{definition}

Random dithering for vectors ($\omega$-quantization) \cite{alistarh2017qsgd}, applied to each column of the matrix, is an example of unbiased compressors that we used in our experiments.

\subsection{Biased compressors.}
    \begin{definition}
    A random operator $\C:~\R^{d\times m} \to \R^{d\times m}$ satisfying
        \begin{equation}\label{eq:quantization_def}
            \EE{\C}{\C(X)} \leq \|X\|_\text{F}, ~~ \EE{\C}{\|\C(X)\|_{\text{F}}^2} \leq (1 - \delta)\|X\|_{\text{F}}^2],
        \end{equation}
        for all $X \in \R^{d\times m}$ is a contractive compressor operator.
    \end{definition}
    In our experiments we use Top-$K$ as a contractive compressor.
\end{document}